\documentclass[11pt, reqno]{amsart}

\usepackage[top=3.75cm, bottom=3cm, left=3cm, right=3cm]{geometry}
\usepackage[utf8]{inputenc}
\usepackage[T1]{fontenc}
\usepackage[dvipsnames]{xcolor}
\usepackage[colorinlistoftodos, textsize = footnotesize]{todonotes}
\usepackage[normalem]{ulem}
\usepackage{listings}

\usepackage{colortbl}
\usepackage{longtable}

\usepackage{hyperref}
\hypersetup{hypertexnames = false, bookmarksdepth = 2, bookmarksopen = true, colorlinks, linkcolor = RedViolet, citecolor = BlueViolet, urlcolor = blue, pdfstartview={XYZ null null 1}, linktocpage = true}

\usepackage{
    amsfonts, amsmath, amsthm, amssymb, mathtools, thmtools,
    tikz-cd, color,
    xparse, xspace,
    enumitem,
    microtype, lmodern
}
\usepackage{calligra}

\usepackage[mathscr]{eucal}

\definecolor{codegreen}{rgb}{0,0.6,0}
\definecolor{codegray}{rgb}{0.5,0.5,0.5}
\definecolor{codepurple}{rgb}{0.58,0,0.82}
\definecolor{backcolour}{rgb}{0.95,0.95,0.92}
\usepackage{prelude/quiver}

\lstdefinestyle{mystyle}{
    backgroundcolor=\color{backcolour},   
    commentstyle=\color{codegreen},
    keywordstyle=\color{magenta},
    numberstyle=\tiny\color{codegray},
    stringstyle=\color{codepurple},
    basicstyle=\ttfamily\footnotesize,
    breakatwhitespace=false,         
    breaklines=true,                 
    captionpos=b,                    
    keepspaces=true,                 
    numbers=left,                    
    numbersep=5pt,                  
    showspaces=false,                
    showstringspaces=false,
    showtabs=false,                  
    tabsize=2
}

\usepackage[capitalise]{cleveref}

\newcounter{todocounter}
\DeclareDocumentCommand\addreference{g}{\stepcounter{todocounter}\todo[color = blue!30]{\thetodocounter. Add reference\IfNoValueF{#1}{: #1}}\xspace}
\DeclareDocumentCommand\checkthis{g}{\stepcounter{todocounter}\todo[color = red!50]{\thetodocounter. Check this\IfNoValueF{#1}{: #1}}\xspace}
\DeclareDocumentCommand\fixthis{g}{\stepcounter{todocounter}\todo[color = orange!50]{\thetodocounter. Fix this\IfNoValueF{#1}{: #1}}\xspace}
\DeclareDocumentCommand\expand{g}{\stepcounter{todocounter}\todo[color = green!50]{\thetodocounter. Expand\IfNoValueF{#1}{: #1}}\xspace}

\DeclareDocumentCommand\commentthis{g}{\todo[color = magenta!50]{Comment:  #1}\xspace}

\numberwithin{equation}{section}

\theoremstyle{plain}
\newtheorem{theorem}{Theorem}[section]
\newtheorem{corollary}[theorem]{Corollary}

\newtheorem{lemma}[theorem]{Lemma}
\newtheorem{proposition}[theorem]{Proposition}

\newtheorem{conjecture}[theorem]{Conjecture}

\theoremstyle{definition}
\newtheorem{definition}[theorem]{Definition}

\newtheorem{remark}[theorem]{Remark}

\declaretheoremstyle[
  spaceabove = 3pt,
  spacebelow = 3pt,
  bodyfont=\normalfont\itshape,
]{alpha}

\makeatletter 
\def\l@subsection{\@tocline{2}{0pt}{3pc}{6pc}{}} 
\makeatother

\makeatletter
\renewcommand{\paragraph}{%
  \@startsection{paragraph}{4}%
  {\z@}{1.5ex \@plus 1ex \@minus .2ex}{-1em}%
  {\normalfont\normalsize\it}%
}
\makeatother

\usepackage{xpatch}
\makeatletter   
\xpatchcmd{\@tocline}
{\hfil\hbox to\@pnumwidth{\@tocpagenum{#7}}\par}
{\ifnum#1<0\hfill\else\dotfill\fi\hbox to\@pnumwidth{\@tocpagenum{#7}}\par}
{}{}
\makeatother 

\AtBeginDocument{%
   \def\MR#1{}
} 

\newcommand{\cA}{\mathscr{A}}
\newcommand{\cB}{\mathscr{B}}
\newcommand{\cC}{\mathscr{C}}

\newcommand{\cE}{\mathscr{E}}

\newcommand{\cI}{\mathcal{I}}

\newcommand{\cL}{\mathcal{L}}

\newcommand{\cO}{\mathcal{O}}

\newcommand{\cT}{\mathcal{T}}

\newcommand{\bC}{\mathbf{C}}
\newcommand{\bP}{\mathbf{P}}
\newcommand{\bF}{\mathbf{F}}
\newcommand{\bZ}{\mathbf{Z}}

\newcommand{\rK}{\mathrm{K}}

\DeclareMathOperator{\Aut}{{Aut}}
\DeclareMathOperator{\Bl}{Bl}

\DeclareMathOperator{\Ext}{Ext}
\DeclareMathOperator{\ext}{ext}

\DeclareMathOperator{\height}{h}
\DeclareMathOperator{\HH}{HH}

\DeclareMathOperator{\Hom}{Hom}

\DeclareMathOperator{\Pic}{Pic}
\DeclareMathOperator{\psheight}{ph}

\DeclareMathOperator{\rheight}{e}

\DeclareMathOperator{\NHH}{NHH}
\DeclareMathOperator{\Db}{{D^b}}

\newcommand*{\SHom}{\mathscr{H}\kern -.5pt om}
\newcommand*{\SExt}{\mathscr{E}\kern -.5pt xt}

\title{A looming of phantoms} 

\author[]{Kimoi Kemboi}

\author[]{Daniel Krashen}

\author[]{Tianle Liu}

\author[]{Yeqin Liu}

\author[]{Eoin Mackall}

\author[]{Svetlana Makarova}

\author[]{Alexander Perry}

\author[]{Antonios-Alexandros Robotis}

\author[]{Sridhar Venkatesh}

\thanks{}

\date{\today}
\allowdisplaybreaks[2]
\begin{document}

\begin{abstract}
    Following Krah's method, we construct new examples of phantom categories as semiorthogonal components of the derived categories of two types of rational surfaces: the blowup of the plane at 11 points in general position, and the blowup of the second Hirzebruch surface at 9 points in general position. We also pose conjectures about the existence of phantom subcategories in the derived categories of other rational surfaces, obtained as the blowups of the other Hirzebruch surfaces. 
\end{abstract}

\maketitle

\tableofcontents

\section{Introduction}

We work over the complex numbers. 
A \emph{phantom} on a smooth projective variety $X$ is a nonzero admissible subcategory $\cC \subset \Db(X)$ of the bounded derived category of coherent sheaves whose Grothendieck group $\mathrm{K}_0(\cC)$ is zero. 
The existence of such categories came as a surprise when the first examples were constructed by Gorchinskiy and Orlov \cite{phantoms-orlov} and by B\"{o}hning, 
Graf von Bothmer, Katzarkov, and Sosna \cite{phantoms-bohning}. The phantoms in \cite{phantoms-bohning} occur on certain general type surfaces which admit maximal length exceptional collections that are not full, while those in \cite{phantoms-orlov} occur on the product of two general type surfaces admitting such exceptional collections. 
The examples in \cite{phantoms-orlov} are in fact \emph{universal phantoms}, meaning that all of their additive invariants vanish, or in other words that their K-motive vanishes. 

Nonetheless, it was expected that on sufficiently simple varieties, such as those admitting a full exceptional collection, phantoms cannot exist \cite[Conjecture 1.10]{kuznetsov-sod-icm}. 
As evidence, Pirozhkov  \cite{phantoms-pirozhkov} proved that phantoms cannot occur on del Pezzo surfaces. 
However, in another twist in the story, Efimov \cite{efimov} proved that \emph{any} universal phantom can be realized on a variety admitting a full exceptional collection (in fact, on an iterated projective bundle).
More recently, Krah \cite{krah} constructed a striking example of a universal phantom on a particularly simple variety: the blowup of $\bP^2$ at $10$ points in general position.  

The emerging picture is that phantom categories are ubiquitous. 
This paper gives further evidence for this viewpoint by constructing more phantoms on rational surfaces following Krah's method. Our first main result handles the blowup of $\bP^2$ at one extra point.  

\begin{theorem}\label{theorem:phantom-11-points}
    Let $X$ be the blowup of $\bP^2$ at $11$ points in general position. 
    Let $H$ be the pullback to $X$ of the hyperplane class on $\bP^2$, let $E_i$ be the exceptional divisors on $X$ for $1\leq i \leq 11$, and set 
    \begin{equation}
    \label{Di-F}
    D_i \coloneqq -3H + \sum_{j=1}^{11} E_j - E_i \quad \text{and} \quad F \coloneqq -10H + 3\sum_{j=1}^{11} E_j.
    \end{equation}
    Then there is a semiorthogonal decomposition 
    \begin{equation}
    \label{sod-10-points}
        \Db(X) = 
        \langle \cC, \cO_X, \cO_X(D_1),\dots, \cO_X(D_{11}),\cO_X(F),\cO_X(2F) \rangle,
    \end{equation}
    where each line bundle is exceptional and the category $\cC$ is a universal phantom. 
\end{theorem}

\begin{remark}
    \label{remark-comparing-phantoms-P2}
    There is another construction of a phantom in the setting of Theorem~\ref{theorem:phantom-11-points}. 
    Let $X'$ denote the blowup of $\bP^2$ at $10$ points in general position, and let $\cC' \subset \Db(X')$ be the phantom constructed by Krah. 
    Let $X$ be the blowup of $X'$ at a point $p$. 
    Then by the blowup formula for derived categories, $\Db(X)$ contains a copy of $\cC'$. 
    On the other hand, if $p \in X'$ is generic, then Theorem~\ref{theorem:phantom-11-points} gives a phantom $\cC \subset \Db(X)$. 
    We compute the Hochschild cohomology group $\HH^2(\cC)$ and observe that it differs from $\HH^2(\cC')$ --- see Proposition~\ref{proposition-HH-11-points} and Remark~\ref{remark-compare-to-krah}. Hence, $\cC$ and $\cC'$ are not equivalent categories, and $\cC$ is a genuinely new example of a phantom. 
\end{remark}

Our second main result produces a phantom on an appropriate blowup of the second Hirzebruch surface $\bF_2 = \bP(\cO_{\bP^1} \oplus \cO_{\bP^1}(2))$. 

\begin{theorem}
    \label{theorem:phantom-f2-9-points}
    Let $X$ be the blowup of $\bF_2$ at $9$ points in general position. 
    Let $C$ and $F$ be the pullbacks to $X$ of the classes of the $-2$ curve and a fiber on $\bF_2$, let $E_i$ be the exceptional divisors on $X$ for $1 \leq i \leq 9$, and set 
    \begin{equation*}
    D_i \coloneqq -4 C - 8F  + 2 \sum_{j=1}^9 E_j  - E_i , 
    \quad 
    G \coloneqq -8C - 17 F + 4 \sum_{j=1}^9 E_j , 
    \quad 
    S \coloneqq -C. 
    \end{equation*}
    Then there is a semiorthogonal decomposition 
    \begin{equation}
        \label{sod-F2-9-points}
            \Db(X) = \langle  \cC, \mathcal{O}_X,\mathcal{O}_X(D_1),\ldots,\mathcal{O}_X(D_9),\mathcal{O}_X(G),\mathcal{O}_X(S+2G),\mathcal{O}_X(S+3G)  \rangle,
    \end{equation} 
    where each line bundle is exceptional and the category $\cC$ is a universal phantom. 
\end{theorem}

\begin{remark}
    \label{remark-comparison-to-krah-F2-intro}
    Similarly to Remark~\ref{remark-comparing-phantoms-P2}, we check via a Hochschild cohomology computation that the phantom from Theorem~\ref{theorem:phantom-f2-9-points} is a genuinely new example, distinct from the Krah's phantom and the one from Theorem~\ref{theorem:phantom-11-points} --- see Proposition~\ref{proposition-HH-11-points} and Remark~\ref{remark-compare-to-krah-F2}. 
\end{remark}

Our third main result, Theorem \ref{theorem:alternatereflection}, constructs a phantom on $\bP^2$ blown up in 10 general points using a different reflection from the one used by Krah. At present, we cannot determine whether or not it agrees with the one constructed by Krah. However, the merit of this result is that it suggests a more general technique for constructing phantoms on rational surfaces, beyond the cases treated by Krah and the first two theorems discussed above.

Finally, in view of Krah's results \cite{krah}
and Theorem \ref{theorem:phantom-f2-9-points}, we formulate a conjecture about the existence of phantom categories on blowups of Hirzebruch surfaces of higher degree. Since $\Bl_1(\bP^2) \cong \bF_1$ and $\Bl_2(\bP^2) \cong \Bl_1(\bF_0)$, Krah's phantom on the blowup of $\bP^2$ in $10$ points shows that $\bF_0$ and $\bF_1$ blown up in $9$ generic points admit phantom categories. In spite of this, results in the literature \cite{BorisovKimoi,phantoms-pirozhkov} suggest that the number of points to be blown up on $\bF_n$ in order for it to support a phantom should depend on $n$; indeed, the number of points in general position needed should equal $h^0(\bF_n,\omega_{\bF_n}^\vee) = 6 + \max\{3,n\}$. Thus, the number of points one expects to blow up to find a phantom as a function of $n$ is $9,9,9,9,10,11,\ldots$. We formulate this more precisely below as Conjecture \ref{Conj:blowupnphantom}.

\subsection{Outline of the construction} 
Let us briefly outline the procedure we use to construct phantoms in the present work, which is based on Krah's technique \cite{krah}. 
We start with a rational surface $X$ admitting a full exceptional exceptional collection of line bundles
\begin{equation}
\label{eq: original collection}
    \Db(X) = \langle \cO(L_1),\dots,\cO(L_d) \rangle.
\end{equation}

\paragraph{Step 1}
Pick an involution $\iota$ of $\Pic (X)$ that preserves the intersection pairing and the canonical divisor $K_X$. This is the key initial input needed to implement Krah's strategy --- in his case, the involution is the negative of reflection along the canonical divisor.

Let $L_i'\coloneqq\iota(L_i)$ denote the images of the divisors $L_i$ under the involution, and consider the new collection
\begin{equation}
\label{eq: reflected collection}
     \cO(L_1'),\dots,\cO(L_d') .
\end{equation}
As a consequence of Riemann--Roch, this new collection is automatically numerically exceptional.

\paragraph{Step 2}
Show that the new collection \eqref{eq: reflected collection} is an honest exceptional collection. 
By Step 1, this reduces to establishing the vanishing of global sections of the divisors $L'_j-L_i'$ and $L'_i-L_j'+K_X$ for $i>j$. To verify this in the situation of Theorem~\ref{theorem:phantom-11-points}, following Krah, we appeal to known cases of the so-called SHGH conjecture (reviewed in \S\ref{section-SHGH}). 
In the situations of \cref{theorem:phantom-f2-9-points} and \cref{theorem:alternatereflection}, we instead use a combination of elementary arguments and Macaulay2 computations, the latter of which are detailed in Appendix \ref{app: mac}. 

\paragraph{Step 3}
The phantom we seek is the orthogonal complement $\cC$ of the new exceptional collection \eqref{eq: reflected collection}. 
By construction $\rK_0(\cC) = 0$, so $\cC$ is a phantom as long as it is nonzero, which holds if and only if the collection~\eqref{eq: reflected collection} is not full. The failure of fullness of an exceptional collection can be detected using Kuznetsov's notion of pseudoheight \cite{kuznetsov-heights}: fullness fails if the pseudoheight is strictly positive (see \cref{lemma:heights-of-exc-coll}). To show that the pseudoheight of \eqref{eq: reflected collection} is positive boils down to showing the vanishing of global sections of the divisors $L'_j-L'_i$ for $i<j$, which we compute as in Step 2.

\begin{remark}
    \label{remark-mattoo}
    If $\cC$ is Krah's phantom, then in the recent paper \cite{mattoo}, 
    Mattoo constructs explicit nonzero objects in $\cC$ with nice properties. In particular, since his objects are nonzero, this gives another another proof of the nonvanishing of $\cC$. 
    We expect that similar methods could be used to produce interesting nonzero objects in the phantom categories constructed in this paper. 
    On the other hand, the approach to nonvanishing of phantoms via the computation of heights of exceptional collections and Hochschild cohomology has the advantage that it can also be used to distinguish the various phantom categories from each other (see Remarks~\ref{remark-comparing-phantoms-P2} and~\ref{remark-comparison-to-krah-F2-intro}). 
\end{remark}

After posting the first version of this article, we learned that Shihao Ma, Yirui Xiong, and Song Yang \cite{MXY-new-phantom}  independently obtained the same phantom as in Theorem \ref{theorem:phantom-11-points}. 

\subsection{Conventions}
We work over the complex numbers, with the exception of Appendices~\ref{sec: semic} and~\ref{app: mac} where we make computations over finite fields to prove vanishing of certain cohomology groups. 
Throughout, $\bF_n$ denotes the Hirzebruch surface of degree $n$, i.e. the projectivization of $\cO_{\bP^1}\oplus \cO_{\bP^1}(n)$. By contrast, $\mathbb{F}_q$ denotes the finite field of order $q$. Somewhat informally, we write $\Bl_i(X)$ to denote the blowup of a surface $X$ in $i$ suitably general points. For a variety $X$, $\Db(X)$ denotes the bounded derived category of coherent sheaves on $X$. We often abbreviate $\dim \Hom$ as $\hom$ and $\dim \Ext^i$ as $\ext^i$ for $i \in \mathbb{Z}$.

\subsection{Acknowledgements} 
The authors are grateful to Pieter Belmans, Shengxuan Liu, and Mahrud Sayrafi for helpful conversations about certain aspects of the paper.

The authors thank the American Mathematical Society for organizing the Mathematics Research Community 2023 “Derived Categories, Arithmetic and Geometry” where this work was initiated. This material is thus based upon work supported by the National Science Foundation under Grant Number DMS 1916439.

We would also like to thank Frederick Fung for assistance with computer calculations, and the National Computational Infrastructure of Australia. Additionally, this research used resources of the shared high performance computing facility at the University of California, Santa Cruz. Support for the UCSC Hummingbird HPC cluster is provided by the Information Technology Services Division and the Office of Research.

During the preparation of this paper, A.P. was partially supported by a Sloan Research Fellowship and NSF grants DMS-2052750 and DMS-2143271. A.R. was partially supported by NSF grant DMS-2503404. S.V. was partially supported by NSF grant DMS-2301463 and by the Simons Collaboration grant \textit{Moduli of Varieties}.

\section{Preliminaries}

\subsection{Phantoms and exceptional collections}

Let $X$ be a smooth projective variety.
We denote by $\Db(X)$ the bounded derived category of coherent sheaves on $X$.
A full triangulated subcategory $\cC \subset \Db(X)$ is \emph{admissible} if the inclusion functor $\cC \hookrightarrow \Db(X)$ admits both right and left adjoints. If $\cC$ is admissible, the natural map $\rK_0(\cC) \to \rK_0(\Db(X)) = \rK_0(X)$ is the inclusion of a direct summand. 

\begin{definition}
Let $\cC \subset \Db(X)$ be a nonzero admissible subcategory.
\begin{enumerate}
    \item If $\rK_0(\cC) =0$, then $\cC$ is said to be a \emph{phantom}.
    \item If for any smooth projective variety $Y$ we have $\rK_0(\cC \boxtimes \Db(Y)) =0$, then $\cC$ is said to be a \textit{universal phantom}. 
\end{enumerate}
Here, $\cC \boxtimes \Db(Y)$ is the smallest triangulated subcategory of $\Db(X\times Y)$ that is both closed under taking direct summands and contains the objects $p^\ast E \otimes q^\ast F$ for all $E \in \cC$, $F \in \Db(Y)$ where $p,q$ are the first and second projection, respectively. 
\end{definition}

In \cite[Definition 1.8]{phantoms-orlov}, phantoms are required to additionally have vanishing Hochschild homology. We drop this condition in our definition above because, since we are working over the field $\mathbb{C}$, it is implied by the vanishing of $\rK_0(\cC)$ (see \cite[Theorem 5.5]{phantoms-orlov}). The following gives a general criterion for an admissible subcategory to be a universal phantom.

\begin{theorem}[{\cite[Proposition 4.4]{phantoms-orlov}}] \label{T:crit_for_universal_phantoms}
    Let $X$ be a smooth projective variety over $\mathbb{C}$ and let $\cC \subset \Db(X)$ be a nonzero admissible subcategory.
    Then the following are equivalent:
    \begin{enumerate}[label =\emph{(\arabic*)}]
        \item $\cC$ is a universal phantom;
        \item $\rK_0(\cC \boxtimes \Db(X)) = 0$;
        \item the $\mathrm{K}$-motive of $\cC$ is zero. 
    \end{enumerate}
\end{theorem}

Recall that a \emph{semiorthogonal decomposition} of $\Db(X)$ is a collection of full triangulated subcategories $\cA_i$, for $i=1,\ldots, n$, written as $\Db(X) = \langle \cA_1, \ldots, \cA_n \rangle$, 
such that:
\begin{enumerate}
    \item $\Hom_{\Db(X)}(A,B) =0$ for all objects $A \in \cA_i$ and $B \in \cA_j$ whenever $i > j$, and
    \item $\Db(X)$ is the smallest full triangulated subcategory containing $\cA_1,\ldots,\cA_n$.
\end{enumerate}
Since $X$ is smooth and projective, the semiorthogonal components $\cA_i \subset \Db(X)$ are all automatically admissible. 
Conversely, any admissible subcategory $\cA \subset \Db(X)$ determines a semiorthogonal decomposition 
\begin{equation*}
\Db(X) = \langle \cA^\perp ,\cA \rangle,
\end{equation*}
where $\cA^\perp\coloneqq\{ F\in \Db(X) \mid  \Hom(A,F) = 0\text{ for all } A\in \cA\}$ is the \textit{right orthogonal} of $\cA$. An important source of admissible subcategories come from exceptional objects. 

\begin{definition} 
Let $X$ be a smooth projective variety. 
\label{D:FEC}
\begin{enumerate} 
\item An object $E \in \Db(X)$ is \emph{exceptional} if 
\begin{equation*}
    \Ext^p(E,E) = \begin{cases} 
    \bC & \text{for } p = 0 , \\ 
    0 & \text{for } p \neq 0.
    \end{cases}
\end{equation*}

\item An \emph{exceptional collection} on $X$ is an ordered collection of exceptional objects $E_1, \ldots, E_n$ in $\Db(X)$ with the property that $\mathrm{RHom}(E_i,E_j) = 0$ for all $i > j$. 

\item An exceptional collection is \emph{full} if $\Db(X)$ is the smallest full triangulated subcategory containing the objects $E_1,\ldots,E_n$. 
\end{enumerate}
\end{definition}

Recall that the Euler characteristic defines a bilinear form on $\mathrm{K}_0(\Db(X)) = \mathrm{K}_0(X)$ by 
\[
    \chi(E,F) = \sum_{i \in \mathbb{Z}} (-1)^i \dim \Ext^i(E,F).
\]
We often use the following weakening of \Cref{D:FEC}.

\begin{definition}
\label{D:NFEC}
    Let $X$ be a smooth projective variety. 
    \begin{enumerate}
        \item An object $E$ of $\Db(X)$ is \emph{numerically exceptional} if $\chi(E,E) = 1$. 
        \item A collection of numerically exceptional objects $E_1,\ldots, E_n$ of $\Db(X)$ is said to be \emph{numerically exceptional} if $\chi(E_i,E_j) = 0$ for all $i > j$.
    \end{enumerate}
\end{definition}

The triangulated category $\langle E \rangle$ generated by an exceptional object inside of $\Db(X)$ is equivalent to $\Db(\bC)$. 
Moreover, any exceptional collection $E_1, \dots, E_n$ on $X$ generates an admissible subcategory $\langle E_1, \ldots, E_n \rangle \subset \Db(X)$, and thus determines 
a semiorthogonal decomposition 
\begin{equation*}
    \Db(X) = \langle \cC, E_1, \dots, E_n \rangle, 
\end{equation*}
where $\cC = \langle E_1, \ldots, E_n \rangle^\perp$ and for simplicity we write $E_i$ for the subcategory $\langle E_i \rangle$ it generates. 
If $\rK_0(X)$ is a free abelian group of rank $n$, then $\rK_0(\cC) = 0$ by the additivity of $\rK_0$ under semiorthogonal decompositions. 
Thus, in this situation, $\cC$ is a phantom as long as it is nonzero. 
This discussion can be amplified to a simple criterion for being a universal phantom. 

\begin{lemma} \label{L:universal_phantoms}
Let $X$ be a smooth projective variety over $\mathbb{C}$ that admits a full exceptional collection of length $n$. Let $\cC \subset \Db(X)$ be a nonzero admissible subcategory that fits into a semiorthogonal decomposition
\begin{equation}
    \Db(X) = \langle \cC, E_1, \ldots, E_n \rangle \label{eq:univ_phantom}
\end{equation}
with $E_1, \ldots, E_n$ an exceptional collection of the same length $n$. Then $\cC$ is a universal phantom.
\end{lemma}

\begin{proof}
By Theorem~\ref{T:crit_for_universal_phantoms}, it suffices to show that $\rK_0(\cC \boxtimes \Db(X)) = 0$. 
Let $F_1, \dots, F_n$ be a full exceptional collection on $X$. It follows --- for instance from {\cite[Proposition 5.1]{Kuznetsov-base-change}} --- that $\Db(X \times X)$ admits a full exceptional collection of length $n^2$, consisting of the objects $F_i \boxtimes F_j$ for $1 \leq i, j \leq n$.  
Thus, $\rK_0(X \times X)$ is a free abelian group of rank $n^2$. 

On the other hand, {\cite[Proposition 5.1]{Kuznetsov-base-change}} also 
implies that \eqref{eq:univ_phantom} induces a semiorthogonal decomposition 
\[
    \Db(X\times X) = \langle \cC \boxtimes \Db(X), E_1 \boxtimes \Db(X), \ldots, E_n \boxtimes \Db(X) \rangle 
\]
and that for each $i = 1, \ldots, n$, we have a semiorthogonal decomposition 
\[
    E_i \boxtimes \Db(X) = \langle E_i \boxtimes F_1, \ldots, E_i \boxtimes F_n \rangle .
\]
It follows that $\cC \boxtimes \Db(X) \subset \Db(X \times X)$ is the right orthogonal to an exceptional collection of length $n^2$. 
As in the discussion preceding this lemma, we thus conclude $\rK_0(\cC \boxtimes \Db(X)) = 0$. 
\end{proof}

\subsection{Hochschild cohomology and heights}
Given a suitably enhanced triangulated category~$\cC$, 
its Hochschild cohomology can be defined as the derived endomorphisms of the identity functor $\mathrm{id}_{\cC}$ in the functor category $\mathrm{Fun}(\cC,\cC)$. 
In this paper, we will be interested in Hochschild cohomology of admissible categories, in which case we can use the following more down-to-earth description of Hochschild cohomology, due to Kuzntesov \cite{kuznetsov-HH}. 

Let  $\Db(X) = \langle \cA ,\cB \rangle$ be a semiorthogonal decomposition, where $X$ is a smooth projective variety. 
Then by \cite[Theorem 7.1]{Kuznetsov-base-change}, there is a distinguished triangle 
\begin{equation}
\label{eq:kernels_of_adm}
    P_\cB \to \Delta_\ast \cO_X \to P_\cA  
\end{equation}
in $\Db(X \times X)$, where $\Delta \colon X \to X \times X$ is the diagonal and the Fourier--Mukai functors $\Phi_{P_{\cA}}, \Phi_{P_{\cB}} \colon \Db(X) \to \Db(X)$ are the projection functors for the semiorthogonal decomposition $\Db(X) = \langle \cA, \cB \rangle$; explicitly, \eqref{eq:kernels_of_adm} is the decomposition of the object $\Delta_{*}\cO_X$ with respect to the semiorthogonal decomposition 
\begin{equation}
\label{sod-XX}
    \Db(X\times X) = \langle \cA \boxtimes \Db(X), \cB \boxtimes \Db(X) \rangle .
\end{equation}
In this situation, we can define the Hochschild cohomology of $\cA$ by 
\begin{equation*}
    \HH^\bullet(\cA) \coloneqq \Hom^\bullet (P_\cA, P_\cA). 
\end{equation*}
By the results of \cite{kuznetsov-HH}, $\HH^\bullet(\cA)$ so defined agrees with the other possible definitions of Hochschild cohomology, and is independent of the realization of $\cA \subset \Db(X)$ as a semiorthogonal component. 
If $\cC \subset \Db(X)$ is an admissible subcategory, then as recalled above it fits into a semiorthogonal decomposition of $\Db(X)$, so this definition of $\HH^\bullet(\cC)$ applies. 
For $\cC = \Db(X)$ the whole category, we use the notation 
\begin{equation*}
    \HH^{\bullet}(X) \coloneqq \HH^{\bullet}(\Db(X)). 
\end{equation*}

Hochschild cohomology is functorial with respect to fully faithful functors that admit an adjoint; see, for instance, \cite[\S4.1]{IHC-CY2} for a general discussion. 
In the situation above where $\cA \hookrightarrow \Db(X)$ is the inclusion of a component in a semiorthogonal decomposition $\Db(X) = \langle \cA , \cB \rangle$ with $X$ smooth and projective, this functoriality can be described explicitly as follows. 
We have $\Hom^{\bullet}(P_{\cB}, P_{\cA}) = 0$ due to the semiorthogonal decomposition~\eqref{sod-XX}, and thus applying $\Hom^\bullet(\Delta_*\cO_X, -)$ to the triangle~\eqref{eq:kernels_of_adm} yields the desired restriction morphism 
\begin{equation}
\label{HH-restriction-morphism}
\HH^\bullet(X) = \Hom^{\bullet}(\Delta_*\cO_X, \Delta_*\cO_X) \to \Hom^{\bullet}(\Delta_*\cO_X, P_{\cA}) \cong \Hom^{\bullet}(P_{\cA}, P_{\cA}) = \HH^{\bullet}(\cA). 
\end{equation}

\begin{definition}[{\cite{kuznetsov-heights}}]
Let  $\Db(X) = \langle \cA ,\cB \rangle$ be a semiorthogonal decomposition, where $X$ is a smooth projective variety. 
Then the \emph{normal Hochschild cohomology} $\NHH^{\bullet}(\cB, X)$ of $\cB$ in $\Db(X)$ is the fiber of the restriction morphism~\eqref{HH-restriction-morphism}, so that there is a distinguished triangle 
\begin{equation}
\label{equation:normal-hochschild-cohomology}
    \NHH^{\bullet}(\cB, X) \to \HH^{\bullet}(X) \to \HH^{\bullet}(\cA). 
\end{equation}
\end{definition}

\begin{remark}
    In \cite{kuznetsov-heights}, Kuznetsov defines normal Hochschild in a more general setting, but in the above situation the general definition is equivalent to ours by \cite[Theorem 3.3]{kuznetsov-heights}. 
\end{remark}

By definition, the normal Hochschild cohomology of $\cB$ in $\Db(X)$ measures the difference between the Hochschild cohomology of $X$ and $\cA$. 
One of the main upshots of \cite{kuznetsov-heights} is that $\NHH^{\bullet}(\cB, X)$ can often be effectively controlled when $\cB$ is generated by an exceptional collection.

\begin{definition}
[{\cite{kuznetsov-heights}}]
\label{definition: heights}
Let $X$ be a smooth projective variety.
Let $E_1, \dots, E_n$ be an exceptional collection on $X$, and let $\cE = \langle E_1, \dots, E_n \rangle \subset \Db(X)$ denote the triangulated subcategory generated by $E_1, \dots, E_n$.  
\begin{enumerate}
    \item The \emph{height} of the exceptional collection $E_1, \dots, E_n$ is
        \[
            \height(E_1, \dots, E_n) \coloneqq \min \{ k \in \mathbb{Z} \mid \NHH^k(\cE,X) \neq 0 \}.
        \]
    \item The \emph{relative height} $\rheight(F,F')$ of two objects $F,F'\in \Db(X)$ is
        \[
        \rheight(F,F') \coloneqq \min \{ k \in \mathbb{Z} \mid \Ext^k(F,F') \neq 0 \}.
        \]
    \item The \emph{pseudoheight} of the exceptional collection $E_1, \dots, E_n$ is
        \begin{align*}
            \psheight(E_1,\ldots, E_n)\coloneqq \min_{{1\le a_0<\cdots< a_p\le n}}\left(\sum_{i=0}^{p-1}\rheight(E_{a_i},E_{a_{i+1}}) + \rheight(E_{a_p},S^{-1}(E_{a_0})) - p \right),
        \end{align*}
    where $S$ denotes the Serre functor of $\Db(X)$.
    \item The \emph{anticanonical pseudoheight} of the exceptional collection $E_1,\ldots, E_n$ is 
    \[
        \psheight_{\rm{ac}}(E_1,\ldots, E_n) \coloneqq \min_{{1\le a_0<\cdots< a_p\le n}}\left(\sum_{i=0}^{p-1}\rheight(E_{a_i},E_{a_{i+1}}) + \rheight(E_{a_p},E_{a_0}\otimes \omega_X^{-1})) - p \right).
    \]
    \end{enumerate}
We note that there is an equality $\psheight(E_1,\ldots, E_n) = \psheight_{\rm{ac}}(E_1,\ldots, E_n) + \dim X$.
\end{definition}

In a precise sense, the height witnesses the fullness of a given exceptional collection while the pseudoheight is an approximation to the height that is more amenable to computation. The utility of these concepts is exemplified by the following lemma.

\begin{lemma}[{\cite[Lemma 4.5 and Proposition 6.1]{kuznetsov-heights}}]\label{lemma:heights-of-exc-coll}Let $E_1,\dots,E_n$ be an exceptional collection on a smooth projective variety $X$. Then $\height(E_1,\dots,E_n) \geq \psheight(E_1,\dots,E_n)$. Additionally, if $\psheight_{\rm{ac}}(E_1,\ldots, E_n)> -\dim X$, then the exceptional collection $E_1,\dots,E_n$ is not full.
\end{lemma}

\subsection{SHGH conjecture}
\label{section-SHGH}

Let $X$ be the blowup of $\bP^2$ at $n$ points in general position. The Segre–Harbourne–Gimigliano–Hirschowitz (SHGH) conjecture provides a precise prediction of the generic behavior of certain linear systems on $X$, see \cite[Conjecture 3.1]{krah}. The following result proves this conjecture for a divisor with ``small multiplicities''.

\begin{theorem}[{\cite[Theorem 34]{MR2325918}}]\label{theorem:SHGH-conjecture-low-multiplicities}
Let $X$ be the blowup of $\bP^2$ at $n$ points in general position. Let $H$ be the pullback to $X$ of the hyperplane class on $\bP^2$, and let $E_i$ for $i=1,\ldots,n$ be the exceptional divisors on $X$. Let $d>0$ and for each of $i=1,\ldots,n$ fix integers $0\leq m_i\leq 11$. Set $D=dH-\sum_{i=1}^n m_iE_i$. Then 
 \[ 
    \dim \mathrm{H}^0(X,\mathcal{O}_X(D))=\max \{0, \chi(X,\mathcal{O}_X(D))\}
\] 
or there exists a $(-1)$-curve $C\subset X$ with $C.D \leq -2$.
\end{theorem}

We will also need the following lemma from \cite{krah} that, together with Theorem \ref{theorem:SHGH-conjecture-low-multiplicities}, will allow us to compute the dimension of the space of global sections of certain divisors on $X$. Using the notation above, we say that a divisor $D=dH-\sum_{i=1}^n m_i E_i$ is in \emph{standard form} if the following are true:

\begin{enumerate}[itemsep=1mm]
\item $d>0$ and $m_i\geq 0$ for all $i=1,\dots ,n$;
\item $d\geq m_1\geq \cdots \geq m_n$; and 
\item $d-m_1-m_2-m_3\geq 0$.
\end{enumerate}

\begin{lemma}[{\cite[Lemma 3.2]{krah}}]\label{lem: krah}
If $D$ is a divisor on $X$ in standard form and  $C\subset X$ is a $(-1)$-curve, then $C.D\geq 0$.
\end{lemma}

\begin{corollary}
\label{Cor:standardformcohomology}
    If $D = dH - \sum_{i=0}^n m_i E_i$ is a divisor on $X$ in standard form with $d>0$, and $0\le m_i\le 11$ for all $i=1,\ldots, n$, then 
    \[
        \dim \mathrm{H}^0(X,\cO_X(D)) = \max\{0,\chi(X,\cO_X(D))\}.
    \]
\end{corollary}

\begin{proof}
    This is immediate from Theorem \ref{theorem:SHGH-conjecture-low-multiplicities} and Lemma \ref{lem: krah}.
\end{proof}

\section{Phantoms on blowups of the projective plane}

\subsection{The blowup of \texorpdfstring{$\bP^2$}{P2} at 11 points}
\label{S:p211points}

In this section, we prove Theorem \ref{theorem:phantom-11-points}. We will freely use the notation therein. 
In particular, $X$ denotes the blowup of $\bP^2$ at $11$ points in general position. 
Orlov's blowup formula for derived categories  \cite[Theorem 4.3]{MR1208153} together with the Beilinson decomposition of $\Db(\bP^2)$ gives a semiorthogonal decomposition 
\begin{equation*}
    \Db(X) = \langle \cO_{E_1}(E_1), \dots, \cO_{E_{11}}(E_{11}), \cO_{X}, \cO_{X}(H), \cO_{X}(2H) \rangle. 
\end{equation*}
By right mutating the objects $\cO_{E_{i}}(E_i)$ past $\cO_{X}$, we obtain a semiorthogonal decomposition 
\begin{equation*} 
     \Db(X) = \langle \cO_X, \cO_X(E_1),\dots, \cO_X(E_{11}),\cO_X(H),\cO_X(2H) \rangle.
\end{equation*}
Observe that $K_X = -3H + \sum_{j=1}^{11} E_j$, and so $K_X^2 = -2$.  Thus, we can consider the transformation on $\Pic(X)$ given by the negative of the reflection along $K_X$:
\begin{equation}
\label{E:Picreflection}
    \begin{split}
        \iota : \Pic(X) &\to \Pic(X), \\
        D &\mapsto -D + 2 \left( \frac{D.K_X}{K_X^2} \right) K_X = -D -(D.K_X) K_X.
    \end{split}
\end{equation}
It is crucial that $\lvert K_X^2 \rvert \le 2$ since it ensures that $\iota$ is an endomorphism of the lattice $\Pic(X)$. Additionally, $\iota$ is an isometric involution of $\Pic(X)$, i.e. $\iota^2 = 1$ and it preserves the intersection product.
Since $\iota(K_X) = K_X$, Riemann--Roch for surfaces and the isometric property then show that for any divisor $D$ on $X$ we have  
\begin{align}
\begin{split}
\label{E:RRsurface}
    \chi(\cO_X(D))  & = \chi(\cO_X) + \tfrac{1}{2} D.(D-K_X)\\
    & = \chi(\cO_X) + \tfrac{1}{2}\iota(D).\iota(D - K_X) \\
    & = \chi(\cO_X(\iota(D))).
\end{split}
\end{align}
Next, observe that
\[ 
    D_i \coloneqq \iota(E_i) = -3H + \sum_{j=1}^{11} E_j - E_i \quad \text{and} \quad F\coloneqq \iota(H) = -10H + 3\sum_{j=1}^{11} E_j. 
\]
Then, \eqref{E:RRsurface} implies that
\begin{equation}\label{equation:transformed-collection}
    \cO_X, \cO_X(D_1),\dots, \cO_X(D_{11}),\cO_X(F),\cO_X(2F) 
\end{equation}
is a numerically exceptional collection of length equal to $\operatorname{rank} \mathrm{K}_0(X)$. In what follows, we abbreviate $\dim \Hom$ as $\hom$ and $\dim \Ext^i$ as $\ext^i$ for $i \in \mathbb{Z}$.

\begin{lemma}\label{lemma:exc-collection-11-pts}
    The collection (\ref{equation:transformed-collection}) is an exceptional collection. 
    Moreover, the line bundles $\cO_X(D_1), \dots, \cO_X(D_{11})$ are completely orthogonal. 
\end{lemma}

\begin{proof}
    To show that (\ref{equation:transformed-collection}) is exceptional, it suffices to check that the appropriate $\Hom$ and $\Ext^2$ groups vanish. This is because the collection is numerically exceptional and 
\[
    \chi(E,F) = \hom(E,F) - \ext^1(E,F) + \ext^2(E,F).
\] 
The lemma thus amounts to the vanishing of the following dimensions, where $1 \leq i,j \leq 11$ and $i \neq j$: 
\begin{align*}
    \hom(\cO_X(2F),\cO_X(F)) &= h^0(-F),\\
    \ext^2(\cO_X(2F),\cO_X(F)) &= h^2(-F) = h^0(K_X + F),\\
    \hom(\cO_X(2F),\cO_X(D_i)) &= h^0(D_i-2F),\\
    \ext^2(\cO_X(2F),\cO_X(D_i)) &= h^2(D_i-2F) = h^0(K_X-D_i+2F),\\
    \hom(\cO_X(2F),\cO_X) &= h^0(-2F),\\
    \ext^2(\cO_X(2F),\cO_X) &= h^2(-2F) = h^0(K_X + 2F),\\
    \hom(\cO_X(F),\cO_X(D_i)) &= h^0(D_i-F),\\
    \ext^2(\cO_X(F),\cO_X(D_i)) &= h^2(D_i-F) = h^0(K_X-D_i+F),\\
    \hom(\cO_X(F),\cO_X) &= h^0(-F),\\
    \ext^2(\cO_X(F),\cO_X) &= h^2(-F) = h^0(K_X + F),\\
    \hom(\cO_X(D_i),\cO_X(D_j)) &= h^0(D_j-D_i),\\
    \ext^2(\cO_X(D_i),\cO_X(D_j)) &= h^2(D_j-D_i) = h^0(K_X-D_j+D_i),\\
    \hom(\cO_X(D_i),\cO_X) &= h^0(-D_i),\\
    \ext^2(\cO_X(D_i),\cO_X) &= h^2(-D_i) = h^0(K_X+D_i),
\end{align*}
The value of $h^0(D)$ in a given line above is zero if the divisor $D$ has negative intersection with the nef divisor $H$ or if $D = D_j - D_i = E_i - E_j$. The remaining cases are:
\begin{equation*}
\begin{gathered}
    -F = 10H - 3\sum_{j=1}^{11} E_j, \quad  -2F = 20H - 6\sum_{j=1}^{11} E_j,\\
    -D_i = 3H -\sum_{j=1}^{11}E_j +E_i,\\
    D_i - F = 7H - 2\sum_{j=1}^{11} E_j- E_i, \quad D_i - 2F = 17H - 5\sum_{j=1}^{11} E_j- E_i
\end{gathered}
\end{equation*}
Observe that if $D = dH - \sum_{j=1}^{11} m_i E_i$ is one of the five divisors above, then $D$ is in standard form.
Thus, by Corollary \ref{Cor:standardformcohomology} we have $h^0(D) = \max\{0, \chi(\cO_X(D))\} = 0$, since $\chi(\cO_X(D)) = 0$ because (\ref{equation:transformed-collection}) is numerically exceptional.
\end{proof}

\begin{lemma}\label{lemma:exts-between-elts-of-exc-coll}
    Let $M_1,M_2$ be any two distinct elements of the collection (\ref{equation:transformed-collection}). Then we have $\Ext^k(M_1,M_2) = 0$ for all $k \neq 2$.
    In particular, we get that $\rheight(M_1,M_2) = 2$ or $\infty$.
    Moreover, $\Ext^k(D_i,D_j) = 0$ for all $i\neq j$, so $\rheight(D_i,D_j)=\infty$.
\end{lemma}

\begin{proof}
    The last claim of the lemma follows from \cref{lemma:exc-collection-11-pts} and \cref{definition: heights}.

    We do the rest in two steps. First, we show that $\Hom$s between elements of the collection are zero. By Lemma \ref{lemma:exc-collection-11-pts}, our collection is exceptional, so we only need to show that the following are zero:
    \begin{align*}
        \hom(\mathcal{O}_X, \mathcal{O}_X(D_i)) &= h^0(D_i), \\
        \hom(\mathcal{O}_X, \mathcal{O}_X(F)) &= h^0(F),\\
        \hom(\mathcal{O}_X, \mathcal{O}_X(2F)) &= h^0(2F), \\
        \hom(\mathcal{O}_X(D_i), \mathcal{O}_X(F)) &= h^0(F - D_i), \\
        \hom(\mathcal{O}_X(D_i), \mathcal{O}_X(2F)) &= h^0(2F - D_i), \\
        \hom(\mathcal{O}_X(F), \mathcal{O}_X(2F)) &= h^0(F).
\end{align*}
    All of these are zero either because the divisor pairs negatively with $H$. 
   
    Next, we show that $\ext^2(M_1,M_2) = \chi(M_1,M_2)$ for $M_1,M_2$ any two distinct elements of the collection (\ref{equation:transformed-collection}). This would imply that $\ext^1(M_1,M_2) = 0$, as required. We have
    \begin{align*}
        \ext^2(\mathcal{O}_X, \mathcal{O}_X(D_i)) &= h^2(D_i) = h^0(K_X - D_i) = h^0(E_i) = 1 = \chi(E_i), \\
        \ext^2(\mathcal{O}_X, \mathcal{O}_X(F)) &= h^2(F) = h^0(K_X - F) = h^0(7H - 2\sum_{j=1}^{11} E_j),\\
        \ext^2(\mathcal{O}_X, \mathcal{O}_X(2F)) &= h^2(2F) = h^0(K_X - 2F) = h^0(17H - 5\sum_{j=1}^{11} E_j), \\
        \ext^2(\mathcal{O}_X(D_i), \mathcal{O}_X(F)) &= h^2(F - D_i) = h^0(K_X - F + D_i) = h^0(4H -\sum_{j=1}^{11} E_j - E_i), \\
        \ext^2(\mathcal{O}_X(D_i), \mathcal{O}_X(2F)) &= h^2(2F - D_i) = h^0(K_X - 2F + D_i) = h^0(14H -4\sum_{j=1}^{11} E_j - E_i), \\
        \ext^2(\mathcal{O}_X(F), \mathcal{O}_X(2F)) &= h^2(F) = h^0(K_X - F) = h^0(7H -2\sum_{j=1}^{11} E_j).
\end{align*}
    For the
    divisors $E_i$,
    we check the equality with the Euler characteristic by hand, whereas all the other divisors are in standard form and fall under the cases of Theorem \ref{theorem:SHGH-conjecture-low-multiplicities}. Thus, Corollary \ref{Cor:standardformcohomology} gives us the equality with the Euler characteristic since the Euler characteristic of each of the above divisors is at least $0$ by Riemann--Roch.
\end{proof}

\begin{lemma}\label{lemma:anticanonical-pseudoheight-3}
    The anticanonical pseudoheight of the exceptional collection (\ref{equation:transformed-collection}) is $1$, or equivalently, its pseudoheight is $3$. In fact, the height of (\ref{equation:transformed-collection}) is $3$.
\end{lemma}

\begin{proof}
    This proof is along the lines of the proof of \cite[Proposition 2.B.3]{krah2024phantoms}. First, observe that $\rheight(\cO_X(D_i),\cO_X(D_j)) = \infty$ by \cref{lemma:exts-between-elts-of-exc-coll}. Now, we have that the anticanonical pseudoheight is the minimum of the following numbers
    \begin{align*}
    & \rheight(\cO_X, \cO_X(-K_X)), \\
    & \rheight(\cO_X, \cO_X(D_i)) + \rheight(\cO_X(D_i), \cO_X(-K_X)) - 1, \\
    & \rheight(\cO_X, \cO_X(F)) + \rheight(\cO_X(F), \cO_X(-K_X)) - 1, \\
    & \rheight(\cO_X, \cO_X(2F)) + \rheight(\cO_X(2F), \cO_X(-K_X)) - 1, \\
    & \rheight(\cO_X(D_i), \cO_X(F)) + \rheight(\cO_X(F), \cO_X(D_i - K_X)) - 1, \\
    & \rheight(\cO_X(D_i), \cO_X(2F)) + \rheight(\cO_X(2F), \cO_X(D_i - K_X)) - 1, \\
    & \rheight(\cO_X(F), \cO_X(2F)) + \rheight(\cO_X(2F), \cO_X(F - K_X)) - 1, \\
    & \rheight(\cO_X, \cO_X(D_i)) + \rheight(\cO_X(D_i), \cO_X(F)) + \rheight(\cO_X(F), \cO_X(-K_X)) - 2, \\
    & \rheight(\cO_X, \cO_X(D_i)) + \rheight(\cO_X(D_i), \cO_X(2F)) + \rheight(\cO_X(2F), \cO_X(-K_X)) - 2, \\
    & \rheight(\cO_X, \cO_X(F)) + \rheight(\cO_X(F), \cO_X(2F)) + \rheight(\cO_X(2F), \cO_X(-K_X)) - 2, \\
    & \rheight(\cO_X(D_i), \cO_X(F)) + \rheight(\cO_X(F), \cO_X(2F)) + \rheight(\cO_X(2F), \cO_X(D_i - K_X)) - 2, \\
    & \rheight(\cO_X, \cO_X(D_i)) 
    + \rheight(\cO_X(D_i), \cO_X(F)) 
    + \rheight(\cO_X(F), \cO_X(2F)) 
    + \rheight(\cO_X(2F), \cO_X(-K_X)) 
    - 3.
    \end{align*}
For $M_1,M_2$ any two distinct elements of the collection (\ref{equation:transformed-collection}), we have $\rheight(M_1,M_2) \geq 2$ by Lemma \ref{lemma:exts-between-elts-of-exc-coll}. Consequently, all the numbers in the above list except the first one are at least $1$. Additionally, for the first number, we have $\rheight(\cO_X, \cO_X(-K_X)) = \rheight(\cO_X,\cO_X(3H - \sum_{j=1}^{11} E_j)) = 1$. This is because $\chi(-K_X) = -1$ and $h^0(- K_X) = \max\{0,\chi(-K_X)\} = 0$ by Corollary \ref{Cor:standardformcohomology} since $-K_X$ is in standard form.
Thus, the minimum of the above list of numbers, which is the anticanonical pseudoheight, is $1$. 

Finally, we note that the anticanonical pseudoheight is achieved on a chain of length $0$, i.e. the anticanonical pseudoheight is equal to $\rheight(\cO_X, \cO_X(-K_X))$. This, along with \cite[Proposition 4.7]{kuznetsov-heights}, implies that the height of the exceptional collection is equal to the pseudoheight.
\end{proof}

We can now prove Theorem \ref{theorem:phantom-11-points}.

\begin{proof}[Proof of Theorem~\ref{theorem:phantom-11-points}]
    We must show that $\cC$ is a universal phantom. 
    By Lemma~\ref{L:universal_phantoms}, it suffices to prove that the collection (\ref{equation:transformed-collection}) is exceptional but not full. We have by Lemma \ref{lemma:exc-collection-11-pts} that the collection is exceptional. Then, Lemma \ref{lemma:heights-of-exc-coll} and Lemma \ref{lemma:anticanonical-pseudoheight-3} together show that the collection is not full.
\end{proof}

The proof of the following lemma is in \cite[Lemma 2.B.5]{krah2024phantoms}. We include it here for the convenience of the reader.

\begin{lemma}
\label{L:h1XTX}
    If $X$ is the blowup of $\bP^2$ in $n \ge 8$ general points then $h^1(X,\cT_X) = 2n - 8$.
\end{lemma}

\begin{proof}
    Let $\pi:X\to \bP^2$ denote the blowup morphism and consider the relative cotangent sequence $0\to \pi^*\Omega_{\bP^2}^1 \to \Omega_X^1\to \Omega^1_{X/\bP^2}\to 0$. Note that $\Omega^1_{X/\bP^2} \cong j_*\Omega_{E/Z}^1$, where $j$ is the inclusion of the exceptional divisor $E = \bigcup_{i=1}^n E_i$ and $Z \subset \bP^2$ is the blown-up locus. Dualizing the relative cotangent sequence gives 
    \[
        0 \to \SHom(j_*\Omega_{E/Z}^1,\cO_X) \to \cT_X\to \pi^* \cT_{\bP^2} \to \SExt^1(j_*\Omega_{E/Z}^1,\cO_X)\to 0
    \]
    where the vanishing of the other $\SExt^1$ terms is by \cite[Propositions III.6.3 \& III.6.7]{hartshorne}. Using the resolution $0\to \cO_X(E) \to \cO_X(2E) \to j_*\cO_E(2E) \cong j_*\Omega_{E/Z}^1\to 0$, one computes that $\SHom(j_*\Omega_{E/Z}^1,\cO_X) = 0$ and $\SExt^1(j_*\Omega_{E/Z}^1,\cO_X) \cong j_*\cO_E(-E)= j_*\bigoplus_{i=1}^n \cO_{E_i}(1)$. Thus, we have a short exact sequence:
    \begin{equation}
    \label{E:SESfortangent}
        0\to \cT_X\to \pi^*\cT_{\bP^2} \to j_*\bigoplus_{i=1}^n \cO_{E_i}(1)\to 0.
    \end{equation}
    Next, global sections of $\cT_X$ can be identified with sections of $\cT_{\bP^2}$ vanishing at the points $p_1,\ldots, p_n$ of $Z$. Since we have chosen the points of $Z$ to be generic, this imposes $n$ linearly independent constraints on $\rm{H}^0(\bP^2,\cT_{\bP^2})$, which is of dimension $8$. Consequently, $h^0(X,\cT_X) = 0$. Finally, by \eqref{E:SESfortangent} we have $h^1(X,\cT_X) = h^0\left(X,j_*\bigoplus_{i=1}^n \cO_{E_i}(1)\right) - h^0(\bP^2,\cT_{\bP^2}) = 2n - 8$. 
\end{proof}

\begin{proposition}\label{proposition-HH-11-points}
    Let $\cC$ be the phantom category in Theorem \ref{theorem:phantom-11-points}. Then $\dim \mathrm{HH}^2(\cC) \geq {14}$.
\end{proposition}

\begin{proof}
    By Lemma \ref{lemma:anticanonical-pseudoheight-3}, the height of the exceptional collection (\ref{equation:transformed-collection}) is $3$. 
    Thus, $\NHH^2(\cE,X) = 0$, where $\cE \subset \Db(X)$ denotes the subcategory generated by the exceptional collection~\eqref{equation:transformed-collection} defining the phantom $\cC$. So, we have the following exact sequence coming from the exact triangle (\ref{equation:normal-hochschild-cohomology}):
    \begin{equation*} 
        \begin{tikzcd}
            0 \arrow[r]& \HH^2(X) \arrow[r, hook]& \HH^2(\cC) \arrow[r] & \NHH^3(\cE,X)\\
        &\mathrm{H}^1(X,\cT_X)\arrow[u,equal,"\rm{HKR}"]\arrow[ur,hook] \ar[r, dash, "\sim"]& \mathbb{C}^{14}&
        \end{tikzcd}
    \end{equation*}
    where ``HKR'' denotes the Hochschild--Kostant--Rosenberg decomposition \cite[p. 140]{HuybrechtsFM} and we have applied Lemma \ref{L:h1XTX} for the bottom isomorphism. 
\end{proof}

\begin{remark}[Comparison with Krah's example]
\label{remark-compare-to-krah}
    Let $X'$ denote the blowup of $\bP^2$ at $10$ points in general position. 
    Krah's result says that there is a semiorthogonal decomposition 
    \begin{equation*}
        \Db(X') = \langle \cC', \cO_{X'}, \cO_{X'}(D'_1), \dots, \cO_{X'}(D'_{10}), \cO_{X'}(F'), \cO_{X'}(2F') \rangle, 
    \end{equation*}
    where the $D'_i$ and $F'$ are certain divisors defined in terms of the hyperplane class and the exceptional divisors by
    similar, but slightly different, formulas to~\eqref{Di-F}. Let $\pi \colon X \to X'$ be the blowup of $X'$ at a point $p \in X$. 
    By Orlov's blowup formula, we obtain a semiorthogonal decomposition 
    \begin{equation}
    \label{E:SOD11points}
        \Db(X) = \langle \cC', \cO_{X}, \cO_{X}(D'_1), \dots, \cO_{X}(D'_{10}), \cO_{X}(F'), \cO_{X}(2F'), \cO_E \rangle , 
    \end{equation}
    where $E$ denotes the exceptional divisor over $p$ and we have still denoted by $D'_i$ and $F'$ their pullbacks to $X$. This shows that the phantom category $\cC'$ embeds in $\Db(X)$ as an admissible subcategory.

    On the other hand, when $p \in X$ is general, Theorem~\ref{theorem:phantom-11-points} also gives a phantom $\cC \subset \Db(X)$. Like $\cC'$, the subcategory $\cC$ is defined as the right orthogonal subcategory to a collection of $13$ exceptional objects. 
    Despite what this might suggest, the categories $\cC$ and $\cC'$ are \emph{not} equivalent. 
    Indeed, the second Hochschild cohomology of these categories differs: $\dim \HH^2(\cC) \geq 14$ by Proposition~\ref{proposition-HH-11-points}, whereas 
    $\dim \HH^2(\cC') = 12$ by \cite[Remark 5.5]{krah}. 
    In particular, $\cC$ is a new phantom. 
    
    In contrast to \cite[Remark 5.5]{krah}, where using these techniques Krah obtains precise values of $\dim \HH^i(\cC')$ for all $i\ge 0$, we can only deduce a lower bound on $\dim \HH^2(\cC)$ in Proposition \ref{proposition-HH-11-points} since Lemma \ref{lemma:anticanonical-pseudoheight-3} does not give vanishing of $\NHH^3(\cE,X)$ (where $\cE$ is the subcategory generated by the exceptional sequence defining the phantom).

Let us also note the following question suggested by the above observations:  
Does the blowup $X_d$ of $\bP^2$ in $d \geq 11$ general points always support a phantom subcategory $\cC_d \subset \Db(X_d)$ such that $\dim \HH^2(\cC_d) = \dim \HH^2(X_d) = 2d-8$.
\end{remark}

\subsection{The blowup of \texorpdfstring{$\bP^2$}{P2} at 10 points: a different reflection}
\label{sec: Bl_10(P2)}

Let $X$ be the blowup of $\bP^2$ at $10$ points in general position.
In this section, we perform a different reflection from the one described in \cite{krah} and show that we still get a phantom category, although we do not know whether this phantom category is equivalent to Krah's.

As above, we start with the full exceptional collection
\[  
    \Db(X) = \langle \cO_X, \cO_X(E_1),\dots, \cO_X(E_{10}),\cO_X(H),\cO_X(2H) \rangle .
\]
This gives us a sequence of divisors $(0,E_1,\dots,E_{10},H,2H)$ in $\Pic X$, to which we will apply a reflection $\iota \colon \Pic X \to \Pic X$ that fixes the canonical divisor $K_X$.
We recall that in \cite{krah}, the reflection was taken to be $D \mapsto -D-2(D.K_X)K_X$.
In our case, we will define an orthogonal projection $P$ onto a plane $\langle v,w \rangle$ and take $\iota(D) \coloneqq D - 2P(D)$, where the pairing is taken with respect to the intersection product.
We choose the following vectors: 

\begin{align*}
    v &\coloneqq -3H + 2(E_1 + E_2) + (E_3+E_4+E_5+E_6+E_7),
    \\
    w &\coloneqq -H + (E_8+E_9+E_{10}).
\end{align*}
With this, we can compute the images of the divisors from the full exceptional collection:
\begin{align*}
    F &\coloneqq \iota(H) = H+18v+26w,
    \\
    D_i &\coloneqq \iota(E_i) =
    \begin{cases}
        E_i + 8v+12w &\text{ for } i=1,2; \\
        E_i + 4v+6w &\text{ for } i=3,\dots,7; \\
        E_i + 6v+8w &\text{ for } i=8,9,10.
    \end{cases}
\end{align*}

\begin{theorem}
\label{theorem:alternatereflection}
    The sequence of line bundles
    \begin{equation}
    \label{reflected EC: Bl_10(P2)}
        \langle \mathcal{O}_X,\mathcal{O}_X(D_1),\ldots,\mathcal{O}_X(D_{10}),
        \cO(F), \cO(2F) \rangle
    \end{equation}
    is exceptional but not full. 
    Moreover, there is a semiorthogonal decomposition 
    \begin{equation*}
        \Db(X) = \langle \cC, \cO_X, \cO_X(D_1), \dots, \cO_X(D_{10}), \cO_X(F), \cO_X(2F) \rangle, 
    \end{equation*}
    where the category $\cC$ is a universal phantom. 
\end{theorem}

\begin{proof}
    As in the proof of Theorem~\ref{theorem:phantom-11-points} above, the second claim follows from the first. 
    We will denote the objects in the exceptional collection as follows:
    \begin{equation}
    \label{reflected EC: Bl_10(P2)abr}
        \langle \cL_1, \dots, \cL_{13} \rangle.
    \end{equation}
    By the Riemann-Roch theorem on surfaces and properties of $\iota$, the sequence is numerically exceptional.
    Therefore, for exceptionality, it is enough to verify that $\Hom(\cL_i,\cL_j)$ and $\Ext^2(\cL_i,\cL_j) \cong \Hom(\cL_j,\cL_i(K_X))$ vanish for $i>j$.
    Further, verifying that $\Hom(\cL_i,\cL_j)=0$ for all $i<j$ will show that the pseudoheight of this collection is at least 1, from which we can conclude that the collection is not full. 
    
    We claim that the global sections of all the relevant line bundles vanish, reasons for which are outlined in \cref{table: Bl_10(P2)}.
    To simplify the notation, we write $E_1' \coloneqq E_1 + E_2$, $E_3'\coloneqq E_3 + \cdots + E_7$, and $E_8' \coloneqq E_8+E_9+E_{10}$.    
\end{proof}


\section{Phantoms on blowups of Hirzebruch surfaces} 

\subsection{The blowup of \texorpdfstring{$\bF_2$}{F2} at 9 points}
\label{sec: f2}

In this section, we prove Theorem~\ref{theorem:phantom-f2-9-points}. 
We work in the setting of the following diagram:
\begin{equation}
\label{E:maps}
    \begin{tikzcd}
        & & \bP^2\\
        \bF_2=V(x_0^2y_1-x_1^2y_0)\arrow[hook, "j"]{r}\arrow[bend left, "\phi_2"]{urr}\arrow[bend right, "\phi_1"]{drr}& \bP^1\times \bP^2\arrow["\pi_2"]{ur}\arrow["\pi_1"]{dr} & \\
        & & \bP^1
    \end{tikzcd}
\end{equation}
Here, we let $x_0,x_1$ be the homogeneous coordinates of $\bP^1$, and $y_0,y_1,y_2$ the homogeneous coordinates of $\bP^2$. The composition $\pi_1\circ j=\phi_1$ realizes the Hirzebruch surface $\bF_2$ as a $\bP^1$-bundle isomorphic to $\bP(\mathcal{O}_{\bP^1}\oplus \mathcal{O}_{\bP^1}(2))$ over $\bP^1$.

We write $C$ for the divisor on $\bF_2$ given by the intersection
\[
    C = \bF_2\cap (\bP^1\times V(y_0,y_1)).
\] 
There is a section of $\phi_1$ with image equal to $C$ given by sending $[x_0:x_1]$ to $([x_0:x_1],[0:0:1])$. Let $F$ be the divisor class of a fiber of $\phi_1$. The divisor class group $\mathrm{Cl}(\bF_2)$ is freely generated by the classes of $C$ and $F$. Moreover, we have 
\[
    \phi_1^*\mathcal{O}_{\bP^1}(1)=\mathcal{O}_{\bF_2}(F)\quad \mbox{and}\quad \phi_2^*\mathcal{O}_{\bP^2}(1)=\mathcal{O}_{\bF_2}(C+2F).
\] 
Thus, regarding $\bF_2$ as a $\bP^1$-bundle, we have $\mathcal{O}_{\bF_2/\bP^1}(1)\cong \mathcal{O}_{\bF_2}(C+2F)$.

By Orlov's projective bundle formula \cite[Corollary 2.7]{MR1208153}, the derived category of $\bF_2$ admits a full exceptional collection 
\[
    \Db(\bF_2)=\langle \mathcal{O}_{\bF_2}, \mathcal{O}_{\bF_2}(F), \mathcal{O}_{\bF_2}(C+2F),\mathcal{O}_{\bF_2}(C+3F)\rangle.
\] 
Let $b:Y \to \bF_2$ be the blowup at a point. Write $s:E\rightarrow Y$ for the inclusion of the exceptional divisor of this blowup. By Orlov's blowup formula \cite[Theorem 4.3]{MR1208153}, there is a full exceptional collection 
\[
    \Db(Y)=\langle s_{*}\mathcal{O}_{E}(-1),\mathcal{O}_{Y},b^*\mathcal{O}_{\bF_2}(F),b^*\mathcal{O}_{\bF_2}(C+2F),b^*\mathcal{O}_{\bF_2}(C+3F)\rangle.
\] 

Applying a right mutation of the first object by the second yields the exceptional collection 
\[
    \Db(Y)=\langle \mathcal{O}_{Y},\mathcal{O}_{Y}(E),b^*\mathcal{O}_{\bF_2}(F),b^*\mathcal{O}_{\bF_2}(C+2F),b^*\mathcal{O}_{\bF_2}(C+3F)\rangle.
\]

Now suppose that $b:X\rightarrow \bF_2$ is the blowup of $\bF_2$ at $9$ sufficiently general points $p_1,\ldots,p_9$. Write $E_i$ for the exceptional divisor in $X$ over the point $p_i$ for all $i=1,\ldots,9$. The Picard group $\mathrm{Pic}(X)$ of $X$ is freely generated by the line bundles 
\[
    \mathcal{O}_X(C) \coloneqq b^*\mathcal{O}_{\bF_2}(C),\quad \mathcal{O}_X(F) \coloneqq b^*\mathcal{O}_{\bF_2}(F),\quad \mbox{and}\quad \mathcal{O}_X(E_i)
\] 
for $i=1,\ldots,9$.\footnote{Note that by a mild abuse of notation, we use $C$ and $F$ to refer to divisors on both $\bF_2$ and $X$.} The intersection numbers of the associated divisors are as follows:
\begin{equation}
\begin{gathered}
C^2=-2,\quad C\cdot F=1,\quad F^2=0\\
C\cdot E_i=0,\quad  F\cdot E_i=0\quad E_i^2=-1,\quad\mbox{and}\quad E_i\cdot E_j=0\quad\text{if $i\neq j$.}
\end{gathered}
\end{equation}
We obtain a full exceptional collection for the blowup $X$ of $\bF_2$ at $p_1,\ldots,p_9$: 
\begin{equation}\label{eq:f2except} 
\Db(X) = \langle \cO_X, \cO_X(E_1),\ldots,\cO_X(E_9),\cO_X(F),\cO_X(C+2F),\cO_X(C+3F)\rangle.
\end{equation}

A canonical divisor on $X$ has class 
\begin{equation}
\label{KX-blowup-F2}
    K_X=-2C-4F+\sum_{i=1}^9E_i
\end{equation}
so that $K_X^2=-1$. Thus, as before, we can consider the isometric involution of $\Pic(X)$ given by the negative of the reflection about $K_X$:
\begin{align*}
    \iota : \Pic(X) &\to \Pic(X)\\
    D &\mapsto -D + 2 \left( \frac{D.K_X}{K_X^2} \right) K_X = -D -2(D.K_X) K_X.
\end{align*}
As in Section \ref{S:p211points}, applying $\iota$ to the exceptional collection \eqref{eq:f2except} produces a collection of sheaves
\begin{equation}
\label{E:ECF2at9points}
\mathcal{O}_X,\mathcal{O}_X(D_1),\ldots,\mathcal{O}_X(D_9),\mathcal{O}_X(G),\mathcal{O}_X(S+2G),\mathcal{O}_X(S+3G),
\end{equation}
which is numerically exceptional by Riemann--Roch, where 
\begin{equation}
\label{eq:f2phantomex}
    D_i\coloneqq\iota(E_i)=2K_X-E_i ,\quad G\coloneqq\iota(F)=-F+4K_X, \text{ and}\quad S\coloneqq\iota(C)=-C.
\end{equation}
Using~\eqref{KX-blowup-F2}, we see that these divisors coincide with the ones defined in the statement of Theorem~\ref{theorem:phantom-f2-9-points}.

\begin{proof}[Proof of Theorem~\ref{theorem:phantom-f2-9-points}]
As in the proof of Theorem~\ref{theorem:phantom-11-points} in \S\ref{S:p211points}, it suffices to prove that the sequence~\eqref{E:ECF2at9points} is exceptional but not full. 

To show that \eqref{E:ECF2at9points} is an exceptional collection, it suffices to check that the appropriate $\Hom$ and $\Ext^2$ groups vanish since the collection is numerically exceptional. Abbreviating $\dim \Hom$ as $\hom$ and $\dim \Ext^2$ as $\ext^2$, we need to show that the following numbers are zero:
\begin{align*}
    \hom(\cO_X(S+3G),\cO_X(S+2G)) &= h^0(-G),\\
    \ext^2(\cO_X(S+3G),\cO_X(S+2G)) &= h^2(-G) = h^0(K_X + G),\\
    \hom(\cO_X(S+3G),\cO_X(G)) &= h^0(-S-2G),\\
    \ext^2(\cO_X(S+3G),\cO_X(G)) &= h^2(-S-2G) = h^0(K_X+S+2G),\\
    \hom(\cO_X(S+3G),\cO_X(D_i)) &= h^0(D_i-S-3G),\\
    \ext^2(\cO_X(S+3G),\cO_X(D_i)) &= h^2(D_i-S-3G) = h^0(K_X-D_i+S+3G),\\
    \hom(\cO_X(S+3G),\cO_X) &= h^0(-S-3G),\\
    \ext^2(\cO_X(S+3G),\cO_X) &= h^2(-S-3G) = h^0(K_X + S+3G),\\
    \hom(\cO_X(S+2G),\cO_X(G)) &= h^0(-S-G),\\
    \ext^2(\cO_X(S+2G),\cO_X(G)) &= h^2(-S-G) = h^0(K_X + S+G),\\
    \hom(\cO_X(S+2G),\cO_X(D_i)) &= h^0(D_i-S-2G),\\
    \ext^2(\cO_X(S+2G),\cO_X(D_i)) &= h^2(D_i-S-2G) = h^0(K_X-D_i+S+2G),\\
    \hom(\cO_X(S+2G),\cO_X) &= h^0(-S-2G),\\
    \ext^2(\cO_X(S+2G),\cO_X) &= h^2(-S-2G) = h^0(K_X + S+2G),\\
    \hom(\cO_X(G),\cO_X(D_i)) &= h^0(D_i-G),\\
    \ext^2(\cO_X(G),\cO_X(D_i)) &= h^2(D_i-G) = h^0(K_X -D_i+G),\\
    \hom(\cO_X(G),\cO_X) &= h^0(-G),\\
    \ext^2(\cO_X(G),\cO_X) &= h^2(-G) = h^0(K_X + G),\\
    \hom(\cO_X(D_i),\cO_X(D_j)) &= h^0(D_j-D_i),\\
    \ext^2(\cO_X(D_i),\cO_X(D_j)) &= h^2(D_j-D_i) = h^0(K_X-D_j+D_i),\\
    \hom(\cO_X(D_i),\cO_X) &= h^0(-D_i),\\
    \ext^2(\cO_X(D_i),\cO_X) &= h^2(-D_i) = h^0(K_X+D_i),
\end{align*}
where $1 \leq i,j \leq 9$, and $i \neq j$. 

To show that \eqref{E:ECF2at9points} is not full, we need to show that the anticanonical pseudoheight of (\ref{equation:transformed-collection}) is at least $-1$. This boils down to showing that the following dimensions are zero:
\begin{align*}
    \hom(\cO_X,\cO_X(D_i)) &= h^0(D_i),\\
    \hom(\cO_X,\cO_X(G)) &= h^0(G),\\
    \hom(\cO_X,\cO_X(S+2G)) &= h^0(S+2G),\\
    \hom(\cO_X,\cO_X(S+3G)) &= h^0(S+3G),\\
    \hom(\cO_X(D_i),\cO_X(D_j)) &= h^0(D_j-D_i),\\
    \hom(\cO_X(D_i),\cO_X(G)) &= h^0(G-D_i),\\
    \hom(\cO_X(D_i),\cO_X(S+2G)) &= h^0(S+2G-D_i),\\
    \hom(\cO_X(D_i),\cO_X(S+3G)) &= h^0(S+3G-D_i),\\
    \hom(\cO_X(G),\cO_X(S+2G)) &= h^0(S+G),\\
    \hom(\cO_X(G),\cO_X(S+3G)) &= h^0(S+2G),\\
    \hom(\cO_X(S+2G),\cO_X(S+3G)) &= h^0(G),
\end{align*}
where $1 \leq i,j \leq 9$, and $i \neq j$. 

Table \ref{table: f2} provides a list of all of the divisors above, barring repeats, as well as a justification for the vanishing of the associated linear system. There are three possible justifications provided:\vspace{3mm}

\begin{itemize}[leftmargin=*]
\item \textit{Degree}: If a divisor $D$ satisfies $(C+3F).D<0$, then $\mathrm{H}^0(X,\mathcal{O}_X(D)) = 0$ since $C+3F$ is nef.\vspace{3mm}
\item \textit{Exceptional}: The vanishing $h^0(E_i-E_j) = 0$ follows from the ideal sequence of $E_j$.\vspace{3mm}
\item \textit{Macaulay2}: The dimension of the cohomology group was computed in Macaulay2 as explained in Appendix \ref{app: mac}. \vspace{3mm}
\end{itemize}
Since all of the necessary cohomology groups vanish, it follows that \eqref{E:ECF2at9points} is exceptional but not full, as claimed.
\end{proof}

In the remainder of the section, we prove some elementary properties about the anticanonical linear system of $\bF_n$, which we use to compute the dimension of $\HH^2(\cC)$, where $\cC$ is the phantom constructed in Theorem \ref{theorem:phantom-f2-9-points}.

\begin{lemma}
    Let $n\ge 0$ be given and consider the Hirzebruch surface $\bF_n$. In this case, $h^0(\bF_n,-K_{\bF_n}) = 6 + \max\{3,n\}$.
\end{lemma}

\begin{proof}
    We consider the projection $p:\bF_n \to \bP^1$. Since $\omega_{\bF_n}^\vee \cong \cO_{\bF_n}(2C) \otimes p^*\cO(n+2)$, the projection formula gives $p_*(\omega_{\bF_n}^\vee) \cong p_*\cO_{\bF_n}(2C)\otimes \cO(n+2)$. On the other hand, by \cite[V.2.8.1]{hartshorne} we know that $p_*\cO_{\bF_n}(2C) = \mathrm{Sym}^2(\cO\oplus \cO(-n)) = \cO \oplus \cO(-n)\oplus \cO(-2n)$. Thus, $p_*(\omega_{\bF_n}^\vee) \cong \cO(n+2)\oplus \cO(2) \oplus \cO(2-n)$ and the result follows.
\end{proof}

In what follows, we put $d(n) = 6 + \max\{3,n\}$.

\begin{lemma}
\label{L:generalpointvanishing}
    One can choose points $p_1,\ldots, p_{d(n)}$ on $\bF_n$ in sufficiently general position such that the surface $\pi:X \to \bF_n$ resulting from blowing up at these points has $\lvert -K_X\rvert =\varnothing$. 
\end{lemma}

\begin{proof}
    Since $K_X = \pi^*K_{\bF_n} + \sum E_i$, it follows that $-K_X = -\pi^*K_{\bF_n} - \sum E_i$ and elements of the linear system $\lvert -K_X\rvert$ can be identified with elements of the linear system $\lvert - K_{\bF_n}\rvert$ which pass through the points $p_1,\ldots, p_{d(n)}$. The condition of passing through $p\in \bF_n$ defines a non-zero (linear) evaluation map $H^0(\bF_n,-K_{\bF_n}) \to \mathbb{C}$. Thus, the general position condition is that $p_1,\ldots, p_{d(n)}$ are chosen such that their evaluation functionals are linearly independent.
\end{proof}

\begin{lemma}
\label{L:pseudoheightofF2at9points}
    The anticanonical pseudoheight of the exceptional collection \eqref{E:ECF2at9points} is $2$, or equivalently, its pseudoheight is $4$. 
\end{lemma}

\begin{proof}
    Similarly to the case of $\bP^2$ blown up in 10 or 11 points, we need to compute certain sheaf cohomology groups. We tabulate the relevant ones in Table \ref{table:F2NHH} of Appendix~\ref{section-tables}. The Euler characteristic $\chi(D)$ of a given divisor can be computed using the Riemann-Roch formula \eqref{E:RRsurface}. The result now follows from mimicking the analysis of Lemma \ref{lemma:anticanonical-pseudoheight-3}. The reader can verify the claim using Table \ref{table:F2NHH}. 
\end{proof}

\begin{remark}
It is worth noting that the main difference between Lemma \ref{lemma:anticanonical-pseudoheight-3} and Lemma \ref{L:pseudoheightofF2at9points} is that in the former $\rheight(\cO_X,\cO_X(-K_X))$ is $1$, while in the latter it is $\infty$. In this way, the situation of Lemma~\ref{L:pseudoheightofF2at9points} is analogous to the case of $\bP^2$ blown up in $10$ points, or equivalently the case of $\bF_1$ blown up in $9$ points.
\end{remark} 

\begin{remark}
    Note that $\chi(-K_X) = 0$ uses the fact that $K_X^2 = -1$, which depends on the number of points blown up, rather than the degree of the Hirzebruch surface. Indeed, if we blow up more than $9$ points, $K_X^2 \le -2$ and it follows that $\chi(-K_X) < 0$. On the other hand, the vanishing $h^0(-K_X) = 0$ depends only on the points blown up being in suitably general position by Lemma \ref{L:generalpointvanishing}. Next, $h^2(-K_X) = h^0(X,\omega_X^{\otimes 2}) = 0$ since the plurigenera of rational surfaces vanish. Consequently, we see that if we blow up at $10$ or more points, $h^1(-K_X) \ne 0$. 
\end{remark}

\begin{lemma}
\label{L:h1TX}
    Let $X$ denote the blowup of the Hirzebruch surface $\bF_n$ at $d(n)$ general points where $n\ge 1$. Then $\HH^2(X) \cong \mathrm{H}^1(X,\cT_X)$. Furthermore, 
    \[
        h^1(X,\cT_X) = 
        \begin{cases}
            11 + n & 1\le n \le 3,\\
            5 + 3n & n \ge 4. 
        \end{cases}
    \]
\end{lemma}

\begin{proof}
    First, $\dim \mathrm{H}^2(X,\cO_X) = h^{0,2}(X) = 0$ for any rational surface. Next, $\Lambda^2 \cT_X \cong \omega_X^\vee$ and the genericity hypothesis on the points $p_1,\ldots, p_{d(n)}$ gives $h^0(X,\omega_X^\vee) = 0$ by Lemma \ref{L:generalpointvanishing}. 
    Thus, the Hochschild--Kostant--Rosenberg decomposition gives an isomorphism $\HH^2(X) \cong \mathrm{H}^1(X,\cT_X)$. 
    
    Now we compute $h^1(X,\cT_X)$. Denote by $f:X\to \bF_n$ the blowup morphism and by $E_1,\ldots, E_{d(n)}$ the exceptional divisors. The short exact sequence 
    \[
        0 \to \cT_X\to f^*\cT_{\bF_n}\to \bigoplus_{i=1}^{d(n)} \cO_{E_i}(1)\to 0
    \]
    transforms under $f_*$ to 
    \[
        0 \to f_* \cT_X \to \cT_{\bF_n}\to \bigoplus_{i=1}^{d(n)} \cT_{\bF_n,p_i}\to 0 
    \]
    where $\cT_{\bF_n,p_i}$ is regarded as a length two sheaf supported at $p_i$. We obtain the long exact sequence that describes deformations of a blown-up surface:
    \[
        0 \to \mathrm{H}^0(X,\cT_X) \to \mathrm{H}^0(\bF_n,\cT_{\bF_n}) \xrightarrow{\rho} (\mathbb{C}^2)^{\oplus d(n)} \to \mathrm{H}^1(X,\cT_X) \to \mathrm{H}^1(\bF_n,\cT_{\bF_n}) \to 0.
    \]
    For $n\ge 1$, using the fact that for $\bF_n$ can be obtained as the blowup at the singular point of the weighted projective space $\bP(1,1,n)$, we see that $\Aut(\bF_n)$ is $3$-transitive for generic choices of points and in particular that $\rho$ maps onto a $6$ dimensional subspace of $(\mathbb{C}^2)^{\oplus d(n)}$. Thus, $h^1(X,\cT_X) = 2d(n) - 6 + h^1(\bF_n,\cT_{\bF_n})$, and the result now follows from \cite[Lemma II.5]{Manettilectures}, which states that $h^1(\bF_n,\cT_{\bF_n}) = n-1$.
\end{proof}

\begin{proposition}
\label{C:HHF2in9points}
    Let $\cC$ be the phantom category from Theorem~\ref{theorem:phantom-f2-9-points}. Then $\dim \HH^2(\cC) = 13$.  
\end{proposition}

\begin{proof}
By Lemmas \ref{lemma:heights-of-exc-coll} and \ref{L:pseudoheightofF2at9points}, it follows that the height of the collection \eqref{E:ECF2at9points} is at least $4$ and thus $\NHH^i(\cE,X) = 0$ for $i\le 3$, where $\cE \subset \Db(X)$ is the subcategory generated by the excpetional collection~\eqref{E:ECF2at9points}. Thus, we have an isomorphism $\HH^2(X) \cong \HH^2(\cC)$ induced by the long exact sequence of the triangle \eqref{equation:normal-hochschild-cohomology}. 
Now the claim follows from Lemma~\ref{L:h1TX}. 
\end{proof}

\begin{remark}[Comparison with other phantoms]
\label{remark-compare-to-krah-F2}
As in Remark~\ref{remark-compare-to-krah}, it follows from Proposition~\ref{C:HHF2in9points} that the phantom category constructed in Theorem \ref{theorem:phantom-f2-9-points} is not equivalent to the one constructed by Krah on the blow up of $\bP^2$ at $10$ points or the one constructed in Theorem \ref{theorem:phantom-11-points} on the blow up of $\bP^2$ at $11$ points.
\end{remark}

\subsection{Conjectures about the case of general \texorpdfstring{$\bF_n$}{Fn}}

Based on the $n=2$ case and Krah's results, we can formulate some guesses about phantom categories on blowups of other Hirzebruch surfaces. The work of Pirozhkov \cite{phantoms-pirozhkov} and the more recent work \cite{BorisovKimoi} have shown that the behavior of the anticanonical linear system plays a key role in the study of phantom categories on surfaces. For instance, \cite[Conjecture 1.3]{BorisovKimoi} predicts that for a smooth projective surface~$X$, the category $\Db(X)$ should have no phantom categories if either of the canonical or anticanonical linear systems of $X$ is nonempty.
The following result is proved in this direction.

\begin{theorem}
[{\cite[Theorem 1.1]{BorisovKimoi}}]
\label{T:BK1.1}
Let $X$ be a smooth complex projective surface with an effective smooth anticanonical divisor $E$ such that the restriction map $\Pic(X) \to \Pic(E)$ is injective. Then $\Db(X)$ has no phantom subcategories.
\end{theorem}

One can apply this result to prove the non-existence of phantom categories on blowups of $\bF_0 = \bP^1\times \bP^1$, $\bF_1$, and $\bP^2$ for any number of points lying in very general position on a smooth anticanonical divisor. More generally, when one studies the anticanonical linear system of a Hirzebruch surface there is a trichotomy of behaviors. Indeed, by \cite[V.2.18(b)]{hartshorne} the anticanonical linear system $\lvert -K_{\bF_n}\rvert$: \vspace{2mm}
\begin{enumerate}
    \item is very ample for $0\le n\le 1$,\vspace{2mm}
    \item is not ample but contains a smooth curve for $n = 2$, and \vspace{2mm}
    \item contains no irreducible curve for $n\ge 3$.\vspace{2mm}
\end{enumerate}

Nevertheless, the case of $n=2$ is beyond the scope of Theorem \ref{T:BK1.1} because the restriction map $\Pic(\bF_2)\to \Pic(E)$ is not injective for any anticanonical divisor $E$, since it is disjoint from the $(-2)$-curve $C$ considered above. However, \cite[Conjecture 1.3]{BorisovKimoi} predicts that $\Db(\bF_n)$ should not have any phantom subcategories since $h^0(\bF_n,\omega_{\bF_n}^\vee) = 6 + \max\{3,n\} > 0$. If $X$ is the blowup of $\bF_n$ at points $p_1,\ldots, p_d$ then $-K_X = -\pi^*K_{\bF_n} - \sum E_i$ and $\lvert -K_X\rvert$ is identified with the sub linear system of $\lvert -K_{\bF_n}\rvert$ consisting of divisors passing through $p_1,\ldots, p_d$. Thus, $\dim \lvert -K_X\rvert \ge \dim \lvert -K_{\bF_n}\rvert - d$ with equality for the points $p_1,\ldots, p_d$ chosen in suitably general position. Thus, we have a pair of conjectures. The first is essentially a restatement of \cite[Conjecture 1.3]{BorisovKimoi} and predicts nonexistence of phantom subcategories:

\begin{conjecture}
\label{Conj:generaln}
    Let $n\ge 0$ be given and let $X$ be a surface obtained from blowing up $\bF_n$ in 
    \begin{enumerate}
        \item $d$ general points with $d < 6 + \max\{3,n\}$, or 
        \item any number of points in general position on an anticanonical divisor of $\bF_n$.
    \end{enumerate}
    In these cases, $\Db(X)$ has no phantom subcategories. 
\end{conjecture}

The second conjecture predicts existence of phantom categories after a critical number of points, which depends on $n$, have been blown up:

\begin{conjecture}
\label{Conj:blowupnphantom}
    If $X$ is obtained from blowing up $\bF_n$ at $d(n) = 6+\max\{3,n\}$ general points then $\Db(X)$ has a phantom subcategory $\cC_n$ orthogonal to an exceptional collection of line bundles of maximal length. Furthermore, the map $\mathrm{H}^1(X,\cT_X) \to \HH^2(\cC_n)$ is an isomorphism so that $\cC_n \not\simeq \cC_m$ for $n\ne m$.
\end{conjecture}

As stated above, a weakened version of Conjecture \ref{Conj:generaln} follows from Theorem \ref{T:BK1.1} when $n=0,1$. For $n\ge 2$, however, new techniques seem to be needed due to the behavior of the anticanonical linear system. We turn now to Conjecture \ref{Conj:blowupnphantom}; since $\Bl_2(\bP^2) \cong \Bl_1(\bF_0)$ and $\Bl_1(\bP^2) \cong \bF_1$, the results of Krah \cite{krah} imply the existence part of the conjecture for $n = 0,1$. Moreover, in the present work we have proven Conjecture \ref{Conj:blowupnphantom} for the case of $n=2$. Similar techniques to the ones used here could be used to check that $\bF_3$ and $\bF_4$ blown up in $9$ and $10$ general points respectively satisfy the conjecture, though we have not attempted this.

\begin{remark}
One can attempt to mimic the proof of Theorem \ref{theorem:phantom-f2-9-points} for $\mathbf{F}_4$ blown-up at $9$ general points. In this setting, we performed computations analogous to those of \cref{sec: app1}, which were used above to show that the collection of Theorem \ref{theorem:phantom-f2-9-points} is exceptional but not full. 

For these computations, we used an embedding of $\mathbf{F}_4$ in $\mathbf{P}^3$ and $9$ random points defined over $\mathbb{F}_{65537}$. The involution $\iota$ used was the negative of the reflection of the canonical divisor as before. We found $h^0(-G)=8$ where $G=\iota(F)$, for a general fiber $F$ of the projective bundle $\mathbf{F}_4/\mathbf{P}^1$, showing that the resulting collection on the blowup of $\bF_4$ is not exceptional over $\mathbb{F}_{65537}$. Although semicontinuity cannot be used to claim the same result for the blowup of $\mathbf{F}_4$ at $9$ general points over $\mathbb{C}$, the lack of vanishing provides some (albeit weak) evidence towards the validity of \cref{Conj:generaln}.
\end{remark}

For $n\ge 5$, different techniques are needed to resolve Conjectures \ref{Conj:generaln} and \ref{Conj:blowupnphantom}. Indeed, the reflections used in the proofs of Theorems~\ref{theorem:phantom-f2-9-points} and~\ref{theorem:phantom-11-points} or in Krah's example will not work for the case of general $n$. 
This is because, if $X$ is a surface obtained by blowing up any Hirzebruch surface in $d$ points, then $K_X^2 = 8 - d$. It follows that for $d\ge 11$ the reflection along $K_X$ is not defined over $\bZ$.

On the other hand, it seems plausible that the use of more general involutions of $\Pic(X)$ --- for instance, like the one used in the proof of Theorem~\ref{theorem:alternatereflection} --- could help resolve Conjecture~\ref{Conj:blowupnphantom} for $n\ge 5$.

\appendix

\section{Semicontinuity and reduction to finite fields}\label{sec: semic}

In this section, we explain how certain cohomology calculations performed over a finite field using Macaulay2 \cite{M2} can be used to deduce corresponding computations over $\mathbb{C}$. There are two main steps to this reduction: reducing from a general member of a family over $\mathbb{C}$ to a particular member, and reducing from $\mathbb{C}$ to $\mathbb{F}_p$ for a prime $p$.

Fix a prime $p$ and let $\mathscr{X}$ be a smooth and projective surface over the $p$-adic integers $\mathbb{Z}_p$. We write $X=\mathscr{X}\times_{\mathbb{Z}_p} \mathbb{Q}_p$ for the generic fiber and $\overline{X}=\mathscr{X}\times_{\mathbb{Z}_p}\mathbb{F}_p$ for the special fiber over $\mathbb{F}_p$. Note that, by choosing an embedding $\mathbb{Q}_p\subset \mathbb{C}$, we get a smooth and projective complex surface $X_{\mathbb{C}}$. In the cases we consider below, we will have either $\mathscr{X}=\mathbf{P}^2_{\mathbb{Z}_p}$ or $\mathscr{X}=\mathbf{F}_{n}/\mathbb{Z}_p$, the $n^{\rm{th}}$ Hirzebruch surface over $\mathbf{P}^1$ for some $n\geq 0$.

Let $\mathscr{I}\subset \mathscr{X}$ be the union of a finite number of disjoint $\mathbb{Z}_p$-sections of the structure morphism of $\mathscr{X}/\mathbb{Z}_p$. By blowing-up along $\mathscr{I}$, we get a smooth and projective surface $\mathscr{Y}=\mathrm{Bl}_{\mathscr{I}}(\mathscr{X})/\mathbb{Z}_p$ and, in a canonical way, we get corresponding schemes \[Y=\mathscr{Y}\times_{\mathbb{Z}_p}\mathbb{Q}_p,\quad \overline{Y}=\mathscr{Y}\times_{\mathbb{Z}_p}\mathbb{F}_p,\quad Y_\mathbb{C}=Y\times_{\mathbb{Q}_p}\mathbb{C}.\] Note that $Y$ (resp.\ $Y_{\mathbb{C}}$) is isomorphic to the blowup of $X$ at $\mathscr{I}\cap X$ (resp.\ the base change of this to $\mathbb{C}$) and $\overline{Y}$ is isomorphic to blowup of $\overline{X}$ at $\mathscr{I}\cap \overline{X}$.

\begin{proposition}
Let $\mathcal{L}$ be a line bundle on $\mathscr{Y}$ and suppose that there exists an integer $i\geq 0$ such that $h^i(\overline{Y},\mathcal{L}|_{\overline{Y}})=0$. Then $h^i(Y_\mathbb{C},\mathcal{L}|_{Y_{\mathbb{C}}})=0$.
\end{proposition}

\begin{proof}
This is an application of \cite[Chapter III, Theorem 12.8]{hartshorne}.
\end{proof}

The complex surface $Y_{\mathbb{C}}$ naturally fits into a family obtained by varying the points of $X_{\mathbb{C}}$ which are blown-up to obtain $Y_{\mathbb{C}}$. To be precise, suppose that $\# \mathscr{I}\cap X=r$ and let $U\subset X_{\mathbb{C}}\times \cdots \times X_{\mathbb{C}}$ denote the open subscheme which is the complement of all partial diagonals. Let $S^0_r$ denote the open subscheme of the $r$th symmetric product $\mathrm{Sym}^r(X_{\mathbb{C}})$ obtained by taking the quotient of $U$ by the induced action of the symmetric group $\Sigma_r$. Let $I\subset X_{\mathbb{C}}\times S^0_r$ denote the incidence subscheme, i.e.\ the reduced subscheme whose $\mathbb{C}$-points have the form $(x,y)$ such that $x$ is contained in the equivalence class of $y$.

By blowing up $I\subset X_{\mathbb{C}}\times S^0_r$, we get a scheme $Z$ over $X_{\mathbb{C}}\times S^0_r$ whose composition with the second projection gives a family of surfaces over $S^0_r$. By construction, the fiber $Z_s$ over a point $s$ of $S^0_r$ is exactly the surface obtained by blowing up $X_{\mathbb{C}}$ in the $r$ distinct points of the equivalence class of the point $s$. Notably, somewhere among the fibers we will find $Y_\mathbb{C}$.

\begin{proposition}
Keep notation as above. Let $\mathcal{L}$ be a line bundle on $Z$, flat over $S^0_r$, and suppose that there exists an integer $i\geq 0$ such that $h^i(Y_\mathbb{C},\mathcal{L}|_{Y_{\mathbb{C}}})=0$. Then there exists an open and dense subset $W\subset S^0_r$ such that for any point $w\in W$, we have $h^i(Z_w,\mathcal{L}|_{Z_w})=0$.
\end{proposition}

\begin{proof}
This is another application of \cite[Chapter III, Theorem 12.8]{hartshorne}.
\end{proof}

\section{Computations with Macaulay2}\label{app: mac}
In this section we explain how one can use Macaulay2 \cite{M2} in order to deduce the vanishing of certain cohomology groups in the main text.

\subsection{Vanishing for Section \ref{sec: f2}}\label{sec: app1}

We keep the notation from Section \ref{sec: f2}, namely we write $\bF_2=V(x_0^2y_1-x_1^2y_0)\subset \bP^1\times \bP^2$ for the Hirzebruch surface whose first projection realizes an isomorphism with the projective bundle $\bP(\mathcal{O}_{\bP^1}\oplus \mathcal{O}_{\bP^1}(2))\to \bP^1$.

Composing the inclusion $\bF_2\subset \bP^1\times \bP^2$ with the Segre embedding $\bP^1\times \bP^2\subset \bP^5$ realizes $\bF_2$ as a quartic scroll in $\bP^5$ embedded by the complete linear system of $\mathcal{O}(C+3F)$, cf.\ \cite[Chapter V, Corollary 2.19]{hartshorne}. Let $z_0,z_1,z_2,z_3,z_4,z_5$ be homogeneous coordinates on $\bP^5$. Then:
\begin{align*}
    \bF_2 & = V(z_0z_4-z_1z_3,z_0z_5-                  z_2z_3,z_1z_5-z_2z_4,
             z_0z_1-z_3^2, 
             z_1^2-z_3z_4, z_1z_2-z_3z_5)\\
C & = V(z_0,z_1,z_3,z_4)\cap \bF_2\\
F & = V(z_1,z_3,z_4,z_5)\cap \bF_2.
\end{align*}

Many of the computations that we want to perform, i.e.\ checking the vanishing of a linear system of a certain line bundle on the blowup $X=\mathrm{Bl}_P(\mathbf{F}_2)$ of $\mathbf{F}_2$ over $\mathbb{C}$ at 9 general points $P=p_1\cup \cdots \cup p_9\subset \mathbf{F}_2$, can be reduced to computations over a finite field, as detailed in Appendix \ref{sec: semic}. We explain here how just one of these computations can be performed: the computation that $\operatorname{rank}\mathrm{H}^0(X,\mathcal{O}_X(-G))=0$.

By semi-continuity (see \cref{sec: semic} for details), it suffices to verify vanishing in the analogous case of the blowup of $\mathbf{F}_2$ over \textit{any finite field} at \textit{any $9$ specific points}. We were able to perform the necessary computation over the ground field $\mathbb{F}_{997}$. The following input was used to initialize the above set-up in Macaulay2 over this field.

\begin{lstlisting}
Fp=ZZ/997;
R=Fp[x_0, x_1, x_2, x_3, x_4, x_5];
I=ideal(x_0*x_4-x_1*x_3, x_0*x_5-x_2*x_3, x_1*x_5-x_2*x_4, x_0*x_1-x_3^2,   x_1^2-x_3*x_4, x_1*x_2-x_3*x_5);
S=R/I;
f=ideal(x_1, x_3, x_4, x_5);
\end{lstlisting}

We randomly selected 9 points on $\mathbf{F}_2$ over this field. The following homogeneous ideals represent the points which were used in our computations.

\begin{lstlisting}[firstnumber = 11]
p1=ideal(x_0-83*x_5, x_1-431*x_5, x_2-644*x_5, x_3-398*x_5, x_4-349*x_5);
p2=ideal(x_0-752*x_5, x_1-134*x_5, x_2-593*x_5, x_3-699*x_5, x_4-281*x_5);
p3=ideal(x_0-810*x_5, x_1-789*x_5, x_2-122*x_5, x_3-546*x_5, x_4-554*x_5);
p4=ideal(x_0-676*x_5, x_1-4*x_5, x_2-984*x_5, x_3-945*x_5, x_4-920*x_5);  
p5=ideal(x_0-20*x_5, x_1-723*x_5, x_2-41*x_5, x_3-730*x_5, x_4-966*x_5);  
p6=ideal(x_0-105*x_5, x_1-530*x_5, x_2-698*x_5, x_3-53*x_5, x_4-315*x_5); 
p7=ideal(x_0-952*x_5, x_1-545*x_5, x_2-750*x_5, x_3-977*x_5, x_4-353*x_5);
p8=ideal(x_0-262*x_5, x_1-661*x_5, x_2-470*x_5, x_3-603*x_5, x_4-589*x_5);
p9=ideal(x_0-557*x_5, x_1-465*x_5, x_2-786*x_5, x_3-588*x_5, x_4-149*x_5);
\end{lstlisting}

In order to compute the linear system of a given divisor, e.g.\ for the divisor $-G$ as in the proof of Theorem \ref{theorem:phantom-f2-9-points}, we let $\mathcal{I}_4$ be the ideal sheaf supported on $p_1\cup\cdots\cup p_9$ in $\mathbf{F}_2$ with multiplicity $4$ at each point. Since there is an equality of divisors 
\[
    8C+17F= -7F + 8(C+3F),
\] 
we can compute the dimension of the linear system $-G$ on $X$ by computing the right-hand-side of the equality below \[\operatorname{rank} \mathrm{H}^0(X,\mathcal{O}_X(-G))=\operatorname{rank} \mathrm{H}^0(\mathbf{F}_2,\mathcal{O}_{\mathbf{F}_2}(-7F)\otimes \mathcal{O}_{\mathbf{F}_2}(8)\otimes \mathcal{I}_4).\] The latter we can calculate by evaluating the Hilbert series of the appropriate graded module, truncating the terms of degree larger than $9$. This can be done with the input below.

\begin{lstlisting}[firstnumber = 20]
Pmult4=intersect(p1^4, p2^4, p3^4, p4^4, p5^4, p6^4, p7^4, p8^4, p9^4);
F7=f^7;
M=Pmult4*F7;
satM=saturate M;
modsatM=module(satM);
hilbertSeries(modsatM, Order => 9)
\end{lstlisting}

This results in the output:
\begin{lstlisting}[firstnumber = 26]
0
ZZ[T]
\end{lstlisting}

The coefficient of this Hilbert series for the term $T^8$ is the $\mathbb{F}_{997}$-dimension of the component of homogeneous degree $8$ of the graded module \texttt{modsatM}. Since the output was the zero polynomial, we obtain the desired vanishing. We note that saturation of the ideal \texttt{M} to \texttt{satM} is necessary for this computation to give the correct result.

Saturation was, by far, the most computationally intensive part of these computations. The computations necessary for \cref{table: f2} were all, eventually, completed on a desktop machine with an AMD Ryzen Threadripper 7960X 24-Core CPU and 64G of RAM. Multiple strategies for computing the saturation were used. The most successful strategy involved computing the saturation of an ideal with respect to each variable with Macaulay2's built-in Eliminate strategy. The following is an example of implementing this strategy to compute \texttt{satM} above.

\begin{lstlisting}[firstnumber = 23]
y_0=lift(x_0, R);
y_1=lift(x_1, R);
y_2=lift(x_2, R);
y_3=lift(x_3, R);
y_4=lift(x_4, R);
y_5=lift(x_5, R);
J=lift(M,R);

M0=saturate(J, y_0, Strategy => Eliminate);
M1=saturate(J, y_1, Strategy => Eliminate);
M2=saturate(J, y_2, Strategy => Eliminate);
M3=saturate(J, y_3, Strategy => Eliminate);
M4=saturate(J, y_4, Strategy => Eliminate);
M5=saturate(J, y_5, Strategy => Eliminate);

satM=intersect(M0,M1,M2,M3,M4,M5)*S;
\end{lstlisting}

Using the Eliminate strategy in this way allowed for each of the computations of Section \ref{sec: f2} to be accomplished in under a day. This often outperformed even parallel strategies, e.g.\ using the msolveSaturate function from the msolve interface in Macaulay2 on 17 threads.

The computations necessary for \cref{table:F2NHH} were achieved on multiple machines. The longest and most difficult computation, using the above instructions, was computing the vanishing of the space $\mathrm{H}^0(X,\mathcal{O}_X(-K_X-S-3G))$. This computation was eventually completed on a machine with a single CPU core and access to 120G of RAM running for 44 hours.

\subsection{Vanishing for Section \ref{sec: Bl_10(P2)}}

We now turn to justifying vanishing for the blowup $X=\mathrm{Bl}_P(\bP^2)$ of $\bP^2$ over $\mathbb{C}$ in $10$ general points $P=p_1\cup \dots \cup p_{10}$, keeping notation from Section \ref{sec: Bl_10(P2)}.
We will provide code to initialize the computations.
\begin{lstlisting}
Fp=ZZ/997;
R=Fp[x_0, x_1, x_2];
p1=ideal(x_0-81*x_1, x_0-491*x_2);
p2=ideal(x_0-752*x_1, x_0-134*x_2);
p3=ideal(x_0-810*x_1, x_0-789*x_2);
p4=ideal(x_0-676*x_1, x_0-4*x_2);
p5=ideal(x_0-20*x_1, x_0-723*x_2);
p6=ideal(x_0-105*x_1, x_0-530*x_2);
p7=ideal(x_0-902*x_1, x_0-545*x_2);
p8=ideal(x_0-212*x_1, x_0-661*x_2);
p9=ideal(x_0-557*x_1, x_0-465*x_2);
p10=ideal(x_0-234*x_1, x_0-632*x_2);

PMult=(m1,m2,m3,m4,m5,m6,m7,m8,m9,m10) -> module(intersect(p1^m1, p2^m2,     p3^m3, p4^m4, p5^m5, p6^m6, p7^m7, p8^m8, p9^m9, p10^m10));
\end{lstlisting}
The function \texttt{PMult} takes as an input a 10-tuple of nonnegative numbers $m_i$ and outputs the graded ideal of functions that vanish at each of the points $p_i$ with multiplicity $m_i$.
Then, running
\begin{lstlisting}[numbers=none]
    hilbertSeries(PMult(m1,m2,m3,m4,m5,m6,m7,m8,m9,m10), Order => d)
\end{lstlisting}
produces the Hilbert series of this graded ideal up to degree $d-1$.
In other words, it outputs the $\mathbb{F}_{997}$-dimensions \[\operatorname{rank} \mathrm{H}^0(X,\cO_X(nH+\sum m_iE_i)) = \operatorname{rank} \mathrm{H}^0(\bP^2,\cI_{p_1}^{m_1}\cdots\cI_{p_{10}}^{m_{10}}(n))\] for all integers $n<d$, where $\cI_{p_i}$ is the ideal sheaf of the point $p_i$ and $H$ a divisor in the class of the pullback of line on $\bP^2$.

For example, we can calculate
\begin{lstlisting}[firstnumber = 15]
    hilbertSeries(PMult(9,8, 3,4,4,4,4, 6,6,6), Order => 19)
\end{lstlisting}
which gives the following output:
\begin{lstlisting}[firstnumber = 16]
    0
    ZZ[T]
\end{lstlisting}
This means that the truncated Hilbert series to degree at most 18 is the zero polynomial, and in particular, $\mathrm H^0(X,\cO_X(18H-9E_1-8E_2-3E_3-4(E_4+\dots+E_7)-6(E_8+E_9+E_{10})) = 0$.

Unlike in \cref{sec: app1}, saturation is automatic for the ideals considered in this section. Indeed, a homogeneous ideal $I$ of a graded polynomial ring $R=k[x_0,\ldots,x_n]$ over a field $k$ is saturated if and only if the irrelevant ideal $R_+$ is not an associated prime for $R/I$, i.e.\ if there is no element $f\in R$ with $fR_+\subset I$. If $J_1,\ldots,J_9\subset R$ are the ideals corresponding to the projective points $p_1,\ldots,p_9\in \bP^2$, then the only associated prime of a power $J_i^{a_i}$ with $a_i>0$ is simply $J_i$. The associated primes of a quotient $R/(J_1^{a_1}\cap \cdots \cap J_9^{a_9})$ with $a_1,\ldots,a_9\geq 0$ are thus a subset of $\{J_1,\ldots,J_9\}$ by \cite[\href{https://stacks.math.columbia.edu/tag/02M3}{Tag 02M3}]{stacks-project}, whence the claim.

\newpage
\section{Tables}
\label{section-tables}

\begin{table}[ht]
\centering
\caption{Vanishing of linear systems on the blowup of $\bP^2$ at 10 generic points.}\label{table: Bl_10(P2)}
\makebox[\textwidth][c]{
\begin{tabular}{|c|c|c|c|}
\hline
Divisor & $(i,j)$ & Expression in terms of $(H,E_i)$ & Reason for vanishing\\ 
\hline
\hline
$-F$ & & $79H-36E'_1-18E'_3-26E'_8$ & Macaulay2\\
\hline
$F+K_X$ & & & degree\\
\hline
$-2F+D_i$ & $i=8,9,10$ & $132H-60E'_1-30E'_3-44E'_8+E_i$ & Macaulay2 \\
\hline
$2F-D_i+K_X$ & $i=8,9,10$ & &degree \\
\hline
$-2F+D_i$ &$i=3,\dots,7$ & $140H-64E'_1-32E'_3+E_i-46E'_8$ & Macaulay2 \\
\hline
$2F-D_i+K_X$ & $i=3,\dots,7$ & &degree \\
\hline
$-2F+D_i$ & $i=1,2$ & $122H-56E'_1+E_i-28E'_3-40E'_8$ & Macaulay2 \\
\hline $2F-D_i+K_X$ & $i=1,2$ & &degree \\
\hline
$-2F$ & & $158H-72E'_1-36E'_3-52E'_8$ & Macaulay2\\
\hline
$2F+K_X$& & & degree\\
\hline $-F+D_i$& $i=8,9,10$ & $53H-24E'_1-12E'_3-18E'_8+E_i$ & Macaulay2 \\
\hline
     $F-D_i+K_X$&
     $i=8,9,10$& &degree \\
\hline
     $-F+D_i$&
     $i=3,\dots,7$& $61H-28E'_1-14E'_3+E_i-20E'_8$ & Macaulay2 \\
\hline 
$F-D_i+K_X$& $i=3,\dots,7$ & &degree \\
\hline
$-F+D_i$&
     $i=1,2$& $43H-20E'_1+E_i-10E'_3-14E'_8$ & Macaulay2 \\
\hline
     $F-D_i+K_X$&
     $i=1,2$
    & &degree \\
\hline
     $D_j-D_i+K_X$&
     $i<j\in\{8,9,10\}$& & degree\\
\hline
     $D_i-D_j$&
     $i=3,\dots,7$, $j=8,9,10$
& $8H-4E'_1-2E'_3+E_i-2E'_8-E_j$ & Macaulay2\\
\hline
     $D_j-D_i+K_X$&
     $i=3,\dots,7$, $j=8,9,10$& & degree\\
\hline
$D_i-D_j$&
     $i=1,2$, $j=8,9,10$& $-10H+5E'_1+E_i+2E'_3+4E'_8-E_j$ & degree\\
\hline
$
     D_j-D_i+K_X$&
     $i=1,2$,  $j=8,9,10$ & $7H-4E'_1-E_i-E'_3-3E'_8+E_j$ & Macaulay2 \\
\hline
$
     -D_i$&
     $i=8,9,10$ & $26H-12E'_1-6E'_3-8E'_8-E_i$ & Macaulay2\\
\hline
$D_i+K_X$&
     $i=8,9,10$ & & degree\\
\hline
     $D_j-D_i+K_X$&
     $i<j\in\{3,\dots,7\}$ & & degree\\
\hline
$D_i-D_j$&
     $i=1,2$, $j=3,\dots,7$ & $-18H+8E'_1+E_i+4E'_3-E_j+6E'_8$ & degree\\
\hline
$D_j-D_i+K_X$&
     $i=1,2$, $j=3,\dots,7$& $15H-7E'_1-E_i-3E'_3+E_j-5E'_8$ & Macaulay2 \\
\hline
$ -D_i$&
     $i=3,\dots,7$& $18H-8E'_1-4E'_3-E_i-6E'_8$ & Macaulay2\\
\hline
$D_i+K_X$&
     $i=3,\dots,7$& & degree\\
\hline
$D_2-D_1+K_X$ & & & degree\\
\hline
$-D_i$&
     $i=1,2$
     & $36H-16E'_1-E_i-8E'_3-12E'_8$ & Macaulay2\\
\hline
$D_i+K_X$&
     $i=1,2$
& & degree\\
\hline
$D_i$&
     $i=1,\dots,10$& & degree\\
\hline
$F$ & & & degree \\
\hline
$2F$ & & &  degree \\
\hline
$D_j-D_i$&
     $i=1,2$, $j=3,\dots,7$& $18H-8E'_1-E_i-4E'_3+E_j-6E'_8$ & Macaulay2 \\
\hline
     $F-D_i$&
     $i=1,\dots,10$&  & degree \\
\hline
     $2F-D_i$&
     $i=1,\dots,10$&  & degree \\
\hline
$D_j-D_i$& $i=3,\dots,7$, $j=8,9,10$
&  & degree \\
\hline
\end{tabular}
}
\end{table}
\newpage

\begin{table}[h]
\centering
\caption{Vanishing of linear systems on the blowup of $\bF_2$ at $9$ generic points.}\label{table: f2}
\begin{tabular}{|c|c|c|}
\hline
Divisor in $(S,G,D_i)$ & Divisor in $(C,F,E_i)$ & Reason for vanishing\\ 
\hline
\hline
$-G$ & $8C+17F-4\sum_{i=1}^9 E_i$ & Macaulay2\\
\hline
$K_X+G$& & degree\\
\hline
$-S-2G$ & $17C+34F-8\sum_{i=1}^9 E_i$ & Macaulay2\\
\hline
$K_X+S+2G$ & & degree\\
\hline
$D_j-S-3G$& $21C+43F-10\sum_{i=1}^9 E_i - E_j$
& Macaulay2 \\
\hline
$K_X-D_i+S+3G$& &degree \\
\hline
$-S-3G$& $25C+51F-12\sum_{i=1}^9 E_i$ & Macaulay2 \\
\hline
$K_X+S+3G$& & degree\\
\hline
$-S-G$& $9C+17F-4\sum_{i=1}^9 E_i$ & Macaulay2\\
\hline
$K_X+S+G$& & degree\\
\hline
$D_j-S-2G$& $13C+26F-6\sum_{i=1}^9 -E_j$ & Macaulay2\\
\hline
$K_X-D_i+S+2G$& & degree\\
\hline
$D_j-G$& $4C+9F-2\sum_{i=1}^9 E_i-E_j$ & Macaulay2\\
\hline
$K_X-D_i+G$& & degree\\
\hline
$D_j-D_i$& $E_i-E_j$ & exceptional \\
\hline
$K_X-D_j+D_i$& & degree\\
\hline
$-D_j$& $4C+8F-2\sum_{i=1}^9 E_i+E_j$ & Macaulay2\\
\hline
$K_X+D_i$& & degree\\
\hline
$D_i$ & & degree\\
\hline
$G$ & & degree\\
\hline
$S+2G$& & degree\\
\hline
$S+3G$ & & degree\\
\hline
$G-D_i$& & degree\\
\hline
$S+2G-D_i$& & degree \\
\hline
$S+3G-D_i$& & degree \\
\hline
$S+G$ &  & degree \\
\hline
$G$ & & degree \\
\hline
\end{tabular}
\end{table}
\newpage
\begin{table}[h]
\centering
\caption{Cohomology of line bundles on the blowup of $\bF_2$ at $9$ generic points, used for computing normal Hochschild cohomology.}\label{table:F2NHH}
\begin{tabular}{|c|c|c|c|c|}
\hline
Divisor & $\chi$ & $h^0$& $h^1$& $h^2$\\ 
\hline
\hline
$D_i$ & 1 &0 &0 & 1\\
\hline
$G$ &2 & 0 & 0 & 2 \\ 
\hline
$G-D_i$ & 1 & 0 & 0  & 1\\
\hline
$S+G$ & 2 & 0 & 0 & 2\\
\hline
$S+2G$ & 4 & 0 & 0 & 4\\
\hline
$S+3G$ & 6 & 0 & 0 & 6 \\
\hline
$S+2G-D_i$ & 3 & 0 & 0 & 3\\
\hline
$S+3G-D_i$ & 5 & 0 & 0 & 5 \\
\hline
$-K_X$ & 0 & 0 & 0 & 0 \\
\hline
$ -K_X - D_i$ & $-2$ & 0 & 2 & 0 \\
\hline
$-K_X-G$ & $-3$ & 0 & 3 & 0 \\
\hline
$-K_X-S-2G$ & $-5$ & 0 & 5 & 0 \\
\hline
$-K_X-S-3G$ & $-7$ & 0 & 7 &0 \\
\hline
\(-K_X-G+D_i\) & $-2$ & 0 & 2 & 0
\\ \hline
\(-K_X-S-2G+D_i\) & \(-4\) & 0 & 4 & 0
\\ \hline
\(-K_X-S-3G+D_i\) & \(-6\) & 0 & 6 & 0
\\ \hline
\(-K_X-S-G\) & \(-3\) & 0 & 3 & 0
\\ \hline
\end{tabular}
\end{table}

\newpage

\bibliographystyle{amsalpha}
\bibliography{bib}

@article{mattoo,
	author = {Mattoo, Amal},
	journal = {arXiv:2510.26107},
	title = {Objects of a phantom on a rational surface},
	year = {2025}}

@article{MXY-new-phantom,
	author = {Shihao Ma and Yirui Xiong and Song Yang},
	journal = {arXiv:2511.07114},
	title = {A new phantom on a rational surface},
	year = {2025}}

@article {IHC-CY2,
    AUTHOR = {Perry, Alexander},
     TITLE = {The integral {H}odge conjecture for two-dimensional
              {C}alabi-{Y}au categories},
   JOURNAL = {Compos. Math.},
  FJOURNAL = {Compositio Mathematica},
    VOLUME = {158},
      YEAR = {2022},
    NUMBER = {2},
     PAGES = {287--333},
}

@article{kuznetsov-HH,
	author = {Kuznetsov, Alexander},
	journal = {arXiv:0904.4330},
	title = {Hochschild homology and semiorthogonal decompositions},
	year = {2009}}

@article {kuznetsov-heights,
    AUTHOR = {Kuznetsov, Alexander},
     TITLE = {Height of exceptional collections and {H}ochschild cohomology
              of quasiphantom categories},
   JOURNAL = {J. Reine Angew. Math.},
  FJOURNAL = {Journal f\"ur die Reine und Angewandte Mathematik. [Crelle's
              Journal]},
    VOLUME = {708},
      YEAR = {2015},
     PAGES = {213--243},
}

@article{efimov,
    Author = {Efimov, Alexander}, 
    Title = {Wall finiteness obstruction for {DG} categories and for algebras over colored {DG} operads}, 
    Journal = {talk slides},
    Year = {2020}, 
}

@inproceedings {kuznetsov-sod-icm,
    AUTHOR = {Kuznetsov, Alexander},
     TITLE = {Semiorthogonal decompositions in algebraic geometry},
 BOOKTITLE = {Proceedings of the {I}nternational {C}ongress of
              {M}athematicians---{S}eoul 2014. {V}ol. {II}},
     PAGES = {635--660},
 PUBLISHER = {Kyung Moon Sa, Seoul},
      YEAR = {2014},
}

@article {phantoms-pirozhkov,
    AUTHOR = {Pirozhkov, Dmitrii},
     TITLE = {Admissible subcategories of del {P}ezzo surfaces},
   JOURNAL = {Adv. Math.},
  FJOURNAL = {Advances in Mathematics},
    VOLUME = {424},
      YEAR = {2023},
     PAGES = {Paper No. 109046, 62},
}

@article {phantoms-orlov,
    AUTHOR = {Gorchinskiy, Sergey and Orlov, Dmitri},
     TITLE = {Geometric phantom categories},
   JOURNAL = {Publ. Math. Inst. Hautes \'Etudes Sci.},
  FJOURNAL = {Publications Math\'ematiques. Institut de Hautes \'Etudes
              Scientifiques},
    VOLUME = {117},
      YEAR = {2013},
     PAGES = {329--349},
}

@article {phantoms-bohning,
    AUTHOR = {B\"ohning, Christian and Graf von Bothmer, Hans-Christian and
              Katzarkov, Ludmil and Sosna, Pawel},
     TITLE = {Determinantal {B}arlow surfaces and phantom categories},
   JOURNAL = {J. Eur. Math. Soc. (JEMS)},
  FJOURNAL = {Journal of the European Mathematical Society (JEMS)},
    VOLUME = {17},
      YEAR = {2015},
    NUMBER = {7},
     PAGES = {1569--1592},
}

@article {krah,
    AUTHOR = {Krah, Johannes},
     TITLE = {A phantom on a rational surface},
   JOURNAL = {Invent. Math.},
  FJOURNAL = {Inventiones Mathematicae},
    VOLUME = {235},
      YEAR = {2024},
    NUMBER = {3},
     PAGES = {1009--1018},
}

@article {MR2325918,
    AUTHOR = {Dumnicki, Marcin and Jarnicki, Witold},
     TITLE = {New effective bounds on the dimension of a linear system in
              {$\mathbf{P}^2$}},
   JOURNAL = {J. Symbolic Comput.},
  FJOURNAL = {Journal of Symbolic Computation},
    VOLUME = {42},
      YEAR = {2007},
    NUMBER = {6},
     PAGES = {621--635},
      ISSN = {0747-7171,1095-855X},
   MRCLASS = {14C20 (13P10)},
  MRNUMBER = {2325918},
MRREVIEWER = {Joaquim\ Ro\'e},
       DOI = {10.1016/j.jsc.2007.01.004},
       URL = {https://doi.org/10.1016/j.jsc.2007.01.004},
}

@article {MR1208153,
    AUTHOR = {Orlov, Dmitri},
     TITLE = {Projective bundles, monoidal transformations, and derived
              categories of coherent sheaves},
   JOURNAL = {Izv. Ross. Akad. Nauk Ser. Mat.},
  FJOURNAL = {Izvestiya Rossiiskoi Akademii Nauk. Seriya Matematicheskaya},
    VOLUME = {56},
      YEAR = {1992},
    NUMBER = {4},
     PAGES = {852--862},
      ISSN = {1607-0046,2587-5906},
   MRCLASS = {14F05 (18E30 18F20)},
  MRNUMBER = {1208153},
MRREVIEWER = {Krzysztof\ Jaczewski},
       DOI = {10.1070/IM1993v041n01ABEH002182},
       URL = {https://doi.org/10.1070/IM1993v041n01ABEH002182},
}

@Misc{M2,
    author = {Grayson, Daniel R. and Stillman, Michael E.},
    title = {Macaulay2, a software system for research in algebraic geometry},
    howpublished = {available at \url{http://www2.macaulay2.com}}
}

@book{hartshorne,
    AUTHOR = {Hartshorne, Robin},
     TITLE = {Algebraic geometry},
    SERIES = {Graduate Texts in Mathematics},
    VOLUME = {No. 52},
 PUBLISHER = {Springer-Verlag, New York-Heidelberg},
      YEAR = {1977},
     PAGES = {xvi+496},
      ISBN = {0-387-90244-9},
   MRCLASS = {14-01},
MRREVIEWER = {Robert\ Speiser},
}

@book{HuybrechtsFM,
    author = {Huybrechts, Daniel},
    title = {Fourier-{M}ukai Transforms in Algebraic Geometry},
    publisher = {Oxford University Press},
    year = {2006},
    month = {04},
    isbn = {9780199296866},
    doi = {10.1093/acprof:oso/9780199296866.001.0001},
    url = {https://doi.org/10.1093/acprof:oso/9780199296866.001.0001},
}

@article {BorisovKimoi,
    AUTHOR = {Borisov, Lev and Kemboi, Kimoi},
     TITLE = {Non-existence of phantoms on some non-generic blowups of the
              projective plane},
   JOURNAL = {Proc. Amer. Math. Soc.},
  FJOURNAL = {Proceedings of the American Mathematical Society},
    VOLUME = {153},
      YEAR = {2025},
    NUMBER = {3},
     PAGES = {963--968},
      ISSN = {0002-9939,1088-6826},
   MRCLASS = {14F08},
       DOI = {10.1090/proc/17105},
       URL = {https://doi.org/10.1090/proc/17105},
}

@article{Manettilectures,
author = {Manetti, Marco},
year = {2004},
month = {01},
pages = {},
title = {Lectures on deformations of complex manifolds. Deformations from differential graded viewpoint},
volume = {24},
journal = {Rendiconti di Matematica e delle sue Applicazioni. Serie VII}
}

@misc{stacks-project,
  author       = {The {Stacks project authors}},
  title        = {The Stacks project},
  howpublished = {\url{https://stacks.math.columbia.edu}},
  year         = {2025},
}

@article {Kuznetsov-base-change,
    AUTHOR = {Kuznetsov, Alexander},
     TITLE = {Base change for semiorthogonal decompositions},
   JOURNAL = {Compos. Math.},
  FJOURNAL = {Compositio Mathematica},
    VOLUME = {147},
      YEAR = {2011},
    NUMBER = {3},
     PAGES = {852--876},
      ISSN = {0010-437X,1570-5846},
   MRCLASS = {14F05 (14A22 18E30)},
  MRNUMBER = {2801403},
MRREVIEWER = {Daniele\ Faenzi},
       DOI = {10.1112/S0010437X10005166},
       URL = {https://doi.org/10.1112/S0010437X10005166},
}

@phdthesis{krah2024phantoms,
  title={Phantoms and Exceptional Collections on Rational Surfaces},
  author={Krah, Johannes},
  year={2024},
  school={Dissertation, Bielefeld, Universit{\"a}t Bielefeld, 2024}
}
\end{document}